\documentclass[12pt,twoside,draft]{cmpart}
\author{ Lysetskyi  T.B., Yeleiko Ya.I.}
\title{Total progeny in almost critical multi-type Galton-Watson processes}

\institute{Ivan Franko National University of Lviv, Lviv, Ukraine (Lysetskyi  T.B.)\\
Ivan Franko National University of Lviv, Lviv, Ukraine (Yeleiko Ya.I.)}

\email{taraslysetskiyy@gmail.com (Lysetskyi  T.B.), yikts@yahoo.com (Yeleiko Ya.I.)}

\enabstract{ Lysetskyi  T.B., Yeleiko Ya.I.}%
{Title of the article in English}%
{We consider a multi-type Galton-Watson branching processes, where the largest in magnitude positive eigenvalue $\rho$ of the first moments matrix is close to unity. Specifically, we examine the random vector representing the number of individuals preceding the generation $n$, often referred to as the total progeny. By conditioning on non-extinction or extinction at current time, and properly normalizing it, we derive the asymptotic distribution for this vector. Similar theorem is derived for the processes with immigration. The behavior of this distribution is primarily influenced by the limit of $n(\rho-1)$ as $n$ tends to infinity and $\rho$ tends to 1.
}
{branching processes, transition phenomena, processes with immigration, total progeny}

\usepackage{bbold}
\UDC{519.2}
\begin{document}

\maketitle

%%% ----------------------------------------------------------------------

\section{Introduction}
To begin with, we  recall the necessary
definitions of some general notions.

Let $X$ be the set of all d-tuples $i = (i_1, ... , i_d )$  with non-negative integer elements.
\begin{definition}
The multi-type Galton-Watson process is a homogeneous vector Markov process $Z_0,Z_1,Z_2,...,$ whose states are vectors in $X$. We shall always assume that $Z_0=1$. We will use the notation $Z_n=\left(Z_{n,1},\ldots,Z_{n,d}\right)', n\in \mathbb{Z}_{+}$. We interpret $Z_{n,i}$, the $i$-th component in $Z_n$ as the number of type $i$ particles in the $n$-th generation. Further, we will refer to the Galton-Watson process as  the G.W. process.
\end{definition}
\begin{definition}
The random vector $Y_n=\left(  Y_{n,1},\ldots,Y_{n,d}\right)', n\in \mathbb{Z}_{+}$, were $Y_{n,k}=\sum_{l=0}^{n}Z_{l,k}, k\in\{1,\ldots,d\}$ will be called the total progeny of a G.W. process. This vector process denotes number of all particles that lived up to the $n$-th generation.
\end{definition}
Denote

$C$: the d -dimensional unit cube, i.e., $C=\{s=(s_1,\ldots, s_d )' \in \mathbb{R}^d: 0\le s_k \le1, k  \in \{1,2,...,d\}$.

$s^{\alpha}=\prod_{i=1}^{d}s_i^{\alpha_i}, \alpha \in X, s \in C.$

$\textbf{1}$: d-dimensional vector of ones, $\textbf{1}= (1,\ldots,1)'.$

$e_{k}$: d-dimensional vector, whose $k$-th entry is one, while others equal zero.

 Let $f_{1,m}(s)$ be a probability generating function (p.g.f.) of $Z_1$,  if the process started from one $m$-th type particle ($Z_0=e_m$) i.e.
\begin{gather*}
f_{1,m}(s)=\sum_{i \in X}p_{m}(i)s^i, s \in C,\\
p_{m}(i) = P\left(Z_n=i\middle|Z_0=e_m\right), i \in X.
\end{gather*}
By $f_{n,m}(s)$, we denote the p.g.f. of $Z_n$ when $Z_0=e_m$.
We will often use vector notations for those p.g.f.'s:
\begin{gather*}
f(s)\equiv f_1(s)\equiv (f_{1,1}(s),\ldots,f_{1,d}(s))',\\
f_{n}(s)=(f_{n,1}(s),...,f_{n,d}(s))', s \in C.
\end{gather*}
 It is well known (see \cite{k3}, p.36), that $f_{n,m}(s)=f_{1,m}(f_{{n-1}}(s))$ or, in the vector form, $f_n(s)=f(f_{n-1}(s)).$ Also, we define the matrix of the first moments of $Z_1$:
\[M=M(f)=\begin{Vmatrix}M_{kl}\end{Vmatrix}_{k,l\in\{1,...,d\}}:M_{kl}=\frac{\partial f_{1,k}(\textbf{1})}{\partial s_l}.\]
By $\rho=\rho(f)$,  we denote the largest positive eigenvalue of $M$ (if it exists). We remind the reader that the G.W. process is called subcritical if $\rho<1$, critical if $\rho=1$, and supercritical if $\rho>1$.

We define $N=\min(n:Z_n=\textbf{0})$, representing the moment of extinction of the process. Note that $\{ N>n\} = \{Z_n>0\}$. 
\begin{definition}
The G.W. process with immigration is defined by a sequense of random vectors $Z^0_{n}=\left(Z^0_{n,1},\ldots,Z^0_{n,d}\right)', n\in \mathbb{Z}_{+}$ with values in $X$, which are determined by the relation
\[Z^0_{n+1}=W^n_1+W^n_2+...+W^n_{Z^0_{n}}+H_{n+1}, n \in \mathbb{N}, \text{ }Z^0_0=H_0,\] 
where $W^n_1,W^n_2,...$ are independent and identically distributed with p.g.f. $f(s)$, the $H_{0},H_{1},...$ are also i.i.d. with p.g.f. $B=B(s)=\sum_{i \in X}p_{0}(i)s^i, s \in C$, and the $W's$ and $H's$ are independent. 
\end{definition}
\begin{definition}
The random vector $Y^0_n=\left(  Y^0_{n,1},\ldots,Y^0_{n,d}\right)', n\in \mathbb{Z}_{+}$, were $Y^0_{n,k}=\sum_{j=0}^{n}Z^0_{j,k}, k\in\{1,2,\ldots,d\}$ will be called the total progeny of a G.W. process with immigration.
\end{definition}

The long-term behavior of the multi-type G.W. processes has been extensively studied for all cases of criticality by numerous authors over the years. For foundational results in G.W. theory, we direct readers to classical treatises \cite{k0} and \cite{k3}. In this paper, our focus lies on the long-term evolution of the total progeny in the near critical multi-type G.W. processes, meaning that $\rho$ is close to $1$. A detailed study of the total progeny in the one-dimensional processes without immigration was provided by Pakes in \cite{k8}. For an overview of basic results on that topic, also refer to Pakes \cite{k8}.
The total progeny in critical processes with immigration in the one-dimensional case was also addressed by Pakes in \cite{k9}. In those works, for the critical processes, Pakes found explicit limits $\lim_{n \to \infty} P\left(Y_n/m_n |Z_n>0\right)$ and $\lim_{n \to \infty} P\left(Y^0_n/m^0_n\right)$, where $m_n$ and $m^0_n$ are normalizing sequences.

The near critical G.W. processes, often termed as the transient phenomena in G.W. processes, were investigated, for example, by Fahady \cite{k2a} in the one-dimensional case and by Quine \cite{k10} in the multi-dimensional case. It is worth noting, that in those cases, limits $\lim_{n \to \infty}P(Z_n/a_n$ $ |Z_n>0)$  and $\lim_{n \to \infty} P\left(Z^0_n/a^0_n\right)$ are the same as in the case $\rho=1$  ($a_n$ and $a^0_n$ are normalizing sequences). Transient phenomena for the total progeny in the one-dimensional case without immigration were studied by Nagaev and Karpenko in \cite{k4}, and by Karpenko in \cite{k5}. More precisely, in \cite{k4}, the limit $\lim_{n \to \infty} P\left(Y_n/m_n |N=n\right)$ was found, while in \cite{k5}, the limit $\lim_{n \to \infty}P \left(Y_n/m_n |Z_n>0\right)$ was obtained. Unlike the case with the limits of $Z_n$ and $Z_n^0$, the limits of 
$Y_n$ and $Y^0_n$ depend on the behavior of $\lim_{n \to \infty, \rho \to 1}n\lvert 1 -\rho \rvert$, and on the direction from which $\rho$ approaches unity.

The asymptotics of G.W. processes and the total progeny in the one-dimensional near-critical processes with immigration were also studied in the Skorokhod $D$ space by Khusanbaev in \cite{k7} and \cite{k7a} (see also the references there). Additionally, large and moderate deviations for the total progeny in the near critical one-dimensional G.W. processes were investigated in \cite{d1}, \cite{d3}, and \cite{d2}.

In in this paper, we extend the results of \cite{k4} and \cite{k5} to the multi-dimensional case (see Theorems \ref{t3} and \ref{t4}). We employ the methods used by Nagaev and Karpenko, establishing multi-dimensional analogies of the lemmas given in \cite{k4} and \cite{k5}, as well as utilizing the methods employed by Quine. Additionally, we derive similar results for the total progeny $Y^0_n$ in the processes with immigration in Theorem \ref{t5}. Finally, in Theorem \ref{t6}, we find the limit $\lim_{n \to \infty}\lim_{m \to \infty} P (Y_n/m_n |Z_{n+m}>0, N < \infty)$, which corresponds to a random variable that serves as the limit of some branching process with immigration, as previously explored by Pakes \cite[p.~186]{k8} for a constant $\rho$ in the one-dimensional case.

\section{Further notations and definitions}

We will use the following notations

$\textbf{0}$: d-dimensional vector of zeros $(0,\ldots,0)'.$

$sw=(s_1w_1,\ldots,s_dw_d)',$ $s/w=(s_1/w_1,\ldots,s_d/w_d)', s,w \in \mathbb{R}^d.$

Throughout the paper, for any two given $d$-dimensional vectors $v=(v_1,\ldots,v_d)'$ and $w=(w_1,\ldots,w_d)'$ notion $v\le (<) w$ means that $v_k\le (<) w_k, k \in \{1,2,...,d\}.$

For the dot product of the vectors $v$ and $w$ we will use the notation $v{'} w$. The notation $ vw'$ denotes the tensor product of the vectors $v$ and $w$.

\begin{remark} \label{r0}
Notation $v{'}w u$ means that we multiply the vector $u$ by the dot product $v{'}w$. If $M$ is a $d \times d$ matrix, then the notation $v'M$  denotes the regular matrix multiplication of the row vector $v'$ with the matrix $M$ and the notation $vM$ indicates that we multiply each element of the first of the matrix $M$ by $v_1$, every element of the second row of  $M$ by $v_2$ and so on.
\end{remark}
Define:
\[M(s)=M(s,f)=\begin{Vmatrix}M_{kl}(s)\end{Vmatrix}_{k,l\in\{1,...,d\}}:M_{kl}(s)=\frac{\partial f_{1,k}(s)}{\partial s_l},\]
\[b^k_{lm}=b^k_{lm}(f)=\frac{\partial f^2_{1,k}(\textbf{1})}{\partial s_l \partial s_m},
c^k_{lmj}=c^k_{lmj}(f)=\frac{\partial f^3_{1,k}(\textbf{1})}{\partial s_l \partial s_m \partial s_j},\]
\[b^k_{lm}(s)=b^k_{lm}(f,s)=\frac{\partial f^2_{1,k}(s)}{\partial s_l \partial s_m}.\]
where $k,l,m,j \in \{1,2,...,d\}, s \in C.$

For a given natural number $U$ and positive constants $a,b,c$,  we define the class 
$\mathcal{K}=\mathcal{K}(a,b,c,U)$ of p.g.f.'s $f(s) \in \mathcal{K}$, for which the following conditions are satisfied:

\begin{equation}\label{eq0}
\begin{array}{l}
A) [M^U(f)]_{kl} \ge a, k,l\in \{1,2,...,d\};\\ 
B) \sum_{k,l,m=1}^{d}b^k_{lm}(f) \ge b;\\ 
C) \sum_{k,l,m,j=1}^{d}c^k_{lmj}(f) \le c.\\ 
\end{array}
\end{equation}

Conditions $A$) and $B$) guarantee that the process is positively regular and not singular (see \cite[pp.~38-39]{k3}), which also implies that the process is irreducible, and the matrix $M$ has a positive eigenvalue $\rho=\rho(f)$ that is greater than any other eigenvalue of $M$ in terms of absolute value (by Perron-Frobenius theorem). Such an eigenvalue is also called the Perron root.  By $\mathcal{K}_{\rho}$ we denote subset of functions $f$ from $\mathcal{K}$ for which the Perron root of $M(f)$ equals to $\rho$.

Let $v=v(f)$ and $u=u(f)$ be the left and the right eigenvectors of $M=M(f)$, respectively, for which

\begin{equation}\label{eq2}
u'v=\sum_{k=1}^{d}u_kv_k=u'\textbf{1}=\sum_{k=1}^{d}u_k=1.
\end{equation}

For matrices $M(s)$, let $\rho_s= \rho_s(f)$ be the Perron root (if it exists); $u_s=u_s(f), v_s=v_s(f)$ be the right and the left eigenvectors corresponding to $\rho_s$, for which the normalization conditions \eqref{eq2} are also satisfied. By $\hat{\rho}_s$, we denote the second largest in magnitude eigenvector of $M(s)$. If $s$ is sufficiently close to $\mathbf{1}$, the Perron root exists, as guaranteed by the Perron-Frobenius theorem. 

We denote the vector of extinction probabilities $\mu =\mu(f)= (\mu_1,\ldots,\mu_d)'$, where  $\mu_j=P(\cup_{n\ge1} Z_n=\textbf{0}|Z_0=e_j)$. The vector $\mu$ is also the smallest of the roots of the equation $f(s)=s$, which exists for $f \in \mathcal{K}$ according to \cite[Theorem 7.1.]{k3} and it also is a limit of an increasing sequence $f_n(\textbf{0})$. It equals $\mathbf{1}$ if $\rho\le1$ and is less than $\mathbf{1}$ if $\rho>1$. 

Specifically for $\mu$, $M(\mu)=M_{\mu}$, $\rho_{\mu}=\rho_{\mu}(f)$  is the Perron root of this matrix,  $v_{\mu}$ and $u_{\mu}$ are the respective left and right eigenvectors. If $\rho \leq 1$, it is evident that $M_{\mu}=M$ and $\rho_{\mu}=\rho$.

 By $\lambda = \lambda(B)= (\lambda_1,\ldots,\lambda_d)'$, we denote the vector of immigration means

\[ \lambda_k=\frac{\partial B(\textbf{1})}{\partial s_k}, k\in \{1,2,...,d\}.\]

We will consider the class of immigration p.g.f.'s $J=J(d_1,d_2)$ which satisfy the following conditions
\begin{equation}\label{eq3}
\begin{array}{l}
A0) B(\textbf{1})=1;\\ 
B0)  \sum_{k=1}^{d}\lambda_k \ge d_1;\\ 
C0)  \sum_{k,l=1}^{d}\frac{\partial B^2(\textbf{1})}{\partial s_k \partial s_l} \le d_2.
\end{array}
\end{equation}

Also, for each $f \in \mathcal{K}$ and $s \in C,$ with $k\in \{1,2,...,d\}$, we define the quadratic form and the weighted sums as follows
\[q_k[s_0,s] = \frac{1}{2}\sum_{l,m=1}^{d}b^k_{lm}(f,s_0)s_ls_m,q_k[\textbf{1},s]=q_k[s]\]
\[Q[s_0,s]=\sum_{k=1}^{d}v_kq_k[s_0,s], Q[\textbf{1},s]=Q[s],Q(s_0)=Q[s_0,u(f)], Q=Q(\textbf{1}).\]

Let 
\begin{equation}\label{eq4}
\pi_0=0, \pi_n=\sum_{j=1}^{n}\rho^{j-2},n\in \mathbb{Z}_{+};
\end{equation}

\begin{equation}\label{eq5}
\psi_n(s)=\frac{\rho^{n}v's}{1+\pi_nQv's}.
\end{equation}

Introduce the following generating functions
\[t_n(s)=(t_{n,1}(s),\ldots,t_{n,d}(s))',\text{ } t_{n,k}(s)=E\left(s^{Y_n}\middle|Z_0=e_k\right),k\in \{1,2,...,d\},\]
\[h_n(s)=(h_{n,1}(s),\ldots,h_{n,d}(s))', \text{ }h_{n,k}(s)=E\left(s^{Y_n}\middle|Z_0=e_k,Z_n=\textbf{0}\right),k\in \{1,2,...,d\}.\]

It is known (see for example \cite[p.~7]{k11}) that both $t_n(s)$ and $h_n(s)$ satisfy functional equations
\begin{equation}\label{eq8}
t_0(s)=s, t_n(s)=sf(t_{n-1}(s)), h_0(s)=\textbf{0}, h_n(s)=sf(h_{n-1}(s)).
\end{equation}

Additionally, according to Pakes \cite{k8}, there exists a unique solution $h^*=h^*(s)$ to the equation in the one-dimensional case
\begin{equation}\label{eq9}
h(s)=sf(h(s)),
\end{equation}
for a fixed $s \in C$. This solution is the limit of a strictly increasing sequence $h_n(s)$ and a strictly decreasing sequence $t_n(s)$. The same result can be easily extended to the multi-dimensional case (Wang proves similar equality in \cite{k11}).

Define
\[g_n(s)=(g^1_n(s),\ldots,g^d_n(s))', \text{where}\]
\[g^k_n(s)=E\left(s^{Y_n}\middle|Z_0=e_k,N=n\right),k\in \{1,2,...,d\}.\]
Functions $g_n(s)$ satisfy the equations 
\begin{equation}\label{eq8a}
g_0(s)=\textbf{0}, g_n(s)=h_{n+1}(s)-h_n(s).
\end{equation}
As for the process with immigration, the p.g.f. of the total progeny, denoted as $\phi_n(s)$ satisfies
\begin{equation}\label{eq10}
\phi_n(s)=\prod_{k=0}^{n-1}B(t_k(s)).
\end{equation}

Let $\sigma=(\frac{1}{n},1-\rho,\textbf{1}-s), \sigma_1=(\frac{1}{n},1-\rho), \sigma_2=(1-\rho,\textbf{1}-s)$. By $o_n(s), o_n, o(s)$ and $o$ we denote scalar or vector quantities that depend on $f \in \mathcal{K}$, such that $\sup_{f \in \mathcal{K}} o_n(s) \to 0$ as $\sigma \to 0$; $\sup_{f \in \mathcal{K}} o_n \to 0$ as $\sigma_1 \to 0$; $\sup_{f \in \mathcal{K}} o(s) \to 0$ as $\sigma_2 \to 0$; and $\sup_{f \in \mathcal{K}} o \to 0$ as $\rho \to 1$.

When studying the near critical processes, we will distinguish two cases: 1) $\lim_{\sigma_1 \to 0}$ $\sup n\lvert1-\rho\rvert \le \hat{c}$; 2) $\lim_{\sigma_1 \to 0} n\lvert 1-\rho\rvert  = +\infty.$ For convenience, instead of mentioning the first (second) case, we will use $i = 1(2)$. Define $\theta_n=nln\rho_{\mu}.$ 

\begin{remark} \label{r1}
For simplicity, we will assume that:

$1)$ all coordinats of $\textbf{1}-s$ converge to $0$ at the same rate as $\sigma_2$ (or $\sigma$) goes to $0$, meaning that there exist positive constants $c_1$ and $c_2$ such that
$c_1 \le \frac{1-s_k}{1-s_m}\le c_2, \sigma_2 (\sigma) \to 0, k,m \in \{1,2,...,d\};$

$2)$ components $1-s_k$ of the vector $\textbf{1}-s$ are at least of the order $O(\min\{\frac{1-\rho}{n}, \frac{1}{n^2}\}).$

Our choice of normalizing sequences in Theorems $(3)-(6)$ will guarantee that those assumptions are satisfied.
\end{remark}

\section{Preliminary results}

We provide Lemma 7 from Quine \cite{k10} in a slightly changed version, which will be crucial for proving Theorems \ref{t3}-\ref{t6}.
Let $P(k,s)$ be a sequence of non-negative primitive $d \times d$ matrices with unit spectral radii and assosiated right and left eigenvectors $v(k,s)$ and $u(k,s)$, respectively, normalized so that $v'(k,s)u(k,s)=u(k,s)\textbf{1}=1$. Set $R(k,s)=u(k,s)v'(k,s)$. 

Let $A(n,m,s)$ be a triple sequence of $d \times d$ matrices satisfying the condition

\[A(n,m,s) \le P(n,s), n,m \in \mathbb{N}, s \in C.\]
Define 
\[B(n,s)=\prod_{k=1}^{n}(P(n,s)-A(n,m,s)).\]
\begin{lemma}[Quine \cite{k10}]\label{l1}
Let the following conditions be satisfied:

i) there exists a null-sequence $\delta_m$ such that 
\[(1-\delta_m)R(n,s)\le P^m(n,s)\le (1+\delta_m)R(n,s);\]

ii) there exists a triple sequence $p_{n,m}(s)$ that is non-increasing for a fixed pair $(n,s)$ with respect to $m$. Additionally, for a fixed $m$, the sequence $p_{n,n-m}(s)$ is null as $n \to \infty$ and $s \to \textbf{1}$, satisfying
\[A(n,m,s) \le p_{n,m}(s)R(n,s), \quad n, m \in \mathbb{N},\]
then for any $\epsilon>0$ and any vector $x \ge \textbf{0}$ for which $B(n)x \ne 0$, there exists natural $N_{\epsilon}$ and $s_{\epsilon} \in C$ such that for $n \ge N_{\epsilon}$ and any $s_{\epsilon}\le s \le \textbf{1}$,
\[-\epsilon<\left[\frac{B(n,s)x}{v'(n,s)B(n,s)x}-u(n,s)\right]_k<\epsilon, k\in \{1,2,...,d\}.\]
\end{lemma}

\begin{remark}\label{r2}
Quine employs only single and double sequences $P(n)$ and $A(n,m)$ for his purposes. Although the proof remains nearly identical, we will not repeat it here.
\end{remark}

Let $A$ be an $d\times d$ matrix. We define $x_A=\max_{\lvert \rho_i \rvert <1}\lvert \rho_i \rvert$, where $\rho_1, ..., \rho_d$ are the eigenvalues of $A$.
\begin{theorem}[Buchanan \cite{k1}]\label{t1}
 Let $F$ be an infinite family of $d\times d$ matrices. The sequences $A^n, A \in F$ converge uniformly if 

i) the eigenvalues of A
are less than 1 in magnitude, except possibly for some eigenvalues equal to 1, 
each of which corresponds to a linear elementary divisor;

ii) there exists $x=x(F)$ such that $x_A\le x < 1, i \in \{1,2,...,d\}$;

iii) there exists $N=N(F)$ such that $x_A^NA$ are uniformly bounded.
\end{theorem}
Define

$X^r = \left\{ i \in X \mid i_k < r, \text{ for } k \in \{1,2,...,d\}, r>0 \right\}$.

$\hat{X}^r = X \setminus X^r$.

\begin{lemma}\label{l2}
For a fixed, $s_0$, $0<s_0 \le \textbf{1}$, there exist constants $0<\eta=\eta(s_0,a,U ), \hat{d}=\hat{d}(s_0,c,U,d_2)<+\infty$, such that for all $f \in \mathcal{K}, B \in J,$ $k,l,m,j \in \{1,2,...,d\}$ and $s \in C_{s_0}: \{ s: \textbf{0}< s_0 \le s \le \textbf{1} \}$ the following inequalities hold:

i) $\rho_s(f) \ge \eta$; ii) $\frac{\hat{\rho}_s(f)}{\rho_s(f)} \le 1-\eta$; iii) $ M_{lm}(s) \le \hat{d}$; iv) $b^k_{lm}(s,f)\le \hat{d}$;  v) $v_{s,k} \le 1/\eta$; vi) $u_{s,k} \ge \eta$; vii) $v_{s,k} \ge \eta$; viii) $\lambda'u \le \lambda'\textbf{1}\le \hat{d}$;
\end{lemma} 
\begin{proof}[Proof]
This lemma was proven by Quine in \cite[Lemma 3]{k10} for $s_0 = \textbf{1}$. In our case, we only prove $(i)$, as the remaining estimates follow similarly to part $(i)$, using Quine's technique.

Let $d(f_1,f_2)=\max_{k \in \{1,2,...,d\}}(\sup_{i \in X}\lvert p_k(i,f_1)-p_k(i,f_2) \rvert)$. Quine proves that with this metric $\mathcal{K}$ is a compact, therefore every sequence $f_n$ has a convergent subsequence $f_{n_k}$.
Let $f_n \to f_{*}$ in metrics $d(\cdot,\cdot)$.

Note that condition $(1(B))$ guarantees the existence of a null sequence $\gamma_r$ which bounds $\sum_{i \in \hat{X}^r} i_lp_{m}(i,f)$ for all $f \in \mathcal{K}$, where $m, l \in \{1,2,...,d\}$.
Therefore
\[\lvert M_{ml}(s,f_{n}) - M_{ml}(s,f_{*})\rvert \le  \sum_{i \in X}i_l \lvert p_{m}(i,f_n)-p_{m}(i,f_{*})\rvert s^{i-e_l} \le  \sum_{i \in X}i_l \lvert p_{m}(i,f_n)-p_{m}(i,f_{*})\rvert\]
\[ \le \sum_{i \in X^r}i_l \lvert p_{m}(i,f_n)-p_{m}(i,f_{*})\rvert + \sum_{i \in \hat{X}^r}i_l( p_{m}(i,f_n)+p_{m}(i,f_{*})).\]
Fix $\epsilon>0$. Choose $r_0$ large enough that $ 4\gamma_{r_0}< \epsilon$ and $n_0$ such that $d(f_n,f_{*})<\frac{\epsilon}{2r_0^{d+1}}$ for all $n>n_0$. Thus, we will have
\begin{equation}\label{eqadd}
\lvert M_{ml}(s,f_{n}) - M_{ml}(s,f_{*})\rvert < \frac{\epsilon}{2}+\frac{\epsilon}{2} = \epsilon.
\end{equation}

For a fixed $s$, there must exist a constant $a(s)$, such that $[M^U(s,f)]_{ml} \ge a(s), k,l\in \{1,2,...,d\}, f \in \mathcal{K}$.
If this were not the case, there must exists a sequence $f_n$, such that 
\begin{equation}\label{eqadd1}
[M^U(s,f_n)]_{ml} \to 0,
\end{equation}
for some $m,l \in \{1,...,d\}$.
However, since $\mathcal{K}$ is compact, there exists convergent subsequence $f_{n_k}$ converging to some $f_{*}$. 
Equation \eqref{eqadd} implies that $[M^U(s,f_{n_k})]_{ml}$ also converges to $[M^U(s,f_{*})]_{ml}$. By \eqref{eqadd1}, it must be zero, which contradicts condition $(1(A))$.
Given that $M_{ml}(s,f)$ are increasing in $s$, $a(s)$ must be non-decreasing. 
Therefore, for all $s \ge s_0$ 
\begin{equation}\label{eqadd2}
  a(s) \ge  a(s_0)>0.
\end{equation}

The product of two compact spaces is compact; therefore, $C_{s_0}\times \mathcal{K}$ is compact. Since eigenvalues are continuous functions of their matrices, then by \eqref{eqadd}, $\rho_s(f)$ attains its lower bound $\eta=\rho_{s^*}(f^*)$ for some point $s^* \in C_{s_0}$ and $f^* \in \mathcal{K}$. By \eqref{eqadd2} and the Perron-Frobenius theorem, we ensure its positivity, thus proving $(i)$.
\end{proof}
Conditions $(i)$ and $(ii)$ of Lemma \ref{l2}, along with condition $B$ from \eqref{eq0}, ensure that the conditions of Theorem \ref{t1} are satisfied for the matrices $\frac{M(s)}{\rho_s}$. As a result, the next lemma is true.
\begin{lemma}\label{l3}
Let $s_0>0$ be a point in $C$. Then there exists a null sequence $\delta_n$ such that for all $f \in \mathcal{K}$ and all $s \in C_{s_0}$ 
\[(1-\delta_n)u_sv'_s \le \frac{M^n(s)}{\rho_s^n} \le (1+\delta_n)u_sv'_s.\]
\end{lemma} 

\begin{lemma}[Karpenko, Nagaev \cite{k4}]\label{l4}
Let $\pi_k(\rho,s)$ be a sequence of non-negative functions (or vector functions) such that
\[\lim_{\sigma_2 \to 0}\sup(\sup_{f \in \mathcal{K}_{\rho}}\pi_k(\rho,s))<\infty,\]
for all $k \in \mathbb{N}$  and 
\[\lim_{\sigma_\to 0}\inf_{f \in \mathcal{K}_{\rho}}\sum_{k=1}^n\pi_k(\rho,s)=\infty,\]
then $\sum_{k=1}^n\pi_k(\rho,s)o_k(s)=\sum_{k=1}^n\pi_k(\rho,s)o_n(s).$
\end{lemma}
In Theorem \ref{t2}, we present necessary results presented in Quine's Theorems 1 and 2.

\begin{theorem}[Quine \cite{k10}]\label{t2}
\begin{equation}\label{eq13aa}
\textbf{1}-f_n(s)=o_n(s), s  \in C,
\end{equation}
\begin{equation}\label{eq13}
\textbf{1}-f_n(s)=\psi_n(\textbf{1}-s)(u+o_n(s)), s  \in C,
\end{equation}
\begin{equation}\label{eq13a}
\textbf{1}-f_n(s)=v'(\textbf{1}-f_n(s))(u+o_n(s)), s  \in C,
\end{equation}
\begin{equation}\label{eq14}
\pi_n P\left(Z_n>0\middle|Z_0=e_m\right)=\rho^nu_mQ^{-1}\left(\textbf{1}+o_n\right).
\end{equation}
\end{theorem}

Now we provide the Taylor expansions for $f(s)$.

The Taylor expansion (or a Taylor-type expansion at the point $\textbf{1}$, see Joffe and Spitzer \cite{k2} and Quine \cite{k10} ) of vectors $f(s)$ at a point $s_0$, where $s,s_0  \in C$, yields
\begin{equation}\label{eq15}
f(s_0)-f(s)=M(s_0)(s_0-s)-E(s)(s_0-s),
\end{equation}
where
\begin{equation}\label{eq15A}
 0\le E(s)\le M(s_0), E(s) \to 0, s\to s_0; s \le t \implies E(t) \le E(s);
\end{equation} 
\begin{equation}\label{eq16}
f(s_0)-f(s)=M(s_0)(s_0-s)-q[s_0,s_0-s]+e_s[s_0,s_0-s]=M(s_0)(s_0-s)-\hat{q}[s_0,s_0-s],
\end{equation}
with
\begin{equation}\label{eq16a}
 0\le e_s[\cdot]\le q[\cdot], e_s[s_0,\cdot] \to \textbf{0} \text{ as } s\to s_0.
\end{equation}

Taylor expansion (or modification of the Joffe and Sitzer method for point  $\textbf{1}$) for gradients 
\[M_k(s)=(M_{k1}(s),\ldots,M_{kd}(s))',k\in \{1,...,d\}\]
gives
\[M_k(s)-M_k(s_0)=q_k(s_0)(s-s_0)-E_k(s)(s-s_0),\]
where $q_k(s_0)=\begin{Vmatrix}q_{klm}(s_0)\end{Vmatrix}_{l,m\in \{1,...,d\}}:q_{klm}=\frac{\partial f^2_{1,k}(s_0)}{\partial s_l\partial s_m}$ and
$0\le E_k(s)\le q_k(s_0), E_k(s) \to 0, s\to s_0,$ or in matrix form
\begin{equation}\label{eq18}
M(s)-M(s_0)=q(s,s_0)-\hat{E}(s,s_0)=\hat{q}(s,s_0),
\end{equation}
where $q(s,s_0)=\begin{Vmatrix}q_{km}(s,s_0)\end{Vmatrix}_{k,m\in \{1,...,d\}}:q_{km}(s,s_0)=[q_k(s_0)(s-s_0)]_m,$

$\hat{E}(s,s_0)=\begin{Vmatrix}\hat{E}_{km}(s,s_0)\end{Vmatrix}_{k,m\in \{1,...,d\}}:\hat{E}_{km}(s,s_0)=[E_k(s)(s-s_0)]_m.$

Finally, we provide the representation for the immigration p.g.f., as given by Quine \cite[p.~441]{k10}
\begin{equation}\label{eq15im}
1-B(s)=\lambda'(\textbf{1}-s)-D'[s](\textbf{1}-s),
\end{equation}
where
\begin{equation}\label{eq15Am}
 0\le D_k[s]\le \frac{d_2}{2}\sum_{l=1}^d(1-s_l), k \in \{1,2,...,d\}.
\end{equation} 

\begin{lemma}\label{l5} 
Let $s_1\to s_2, s_1, s_2 \in C_{s_0}$, then there exists constant $c_m$, such that for all $f \in \mathcal{K}$
\begin{equation}\label{eqh}
\lvert v'_{s_1}M(s_2)u_{s_1}-\rho_{s_2} \rvert \le c_m \lvert s_1-s_2 \rvert^2.
\end{equation}
\end{lemma}
\begin{proof}[Proof]
In view of \eqref{eq2}, we can rewrite $v'_{s_1}M(s_2)u_{s_1}-\rho_{s_1}$ as
\[v'_{s_1}M(s_2)u_{s_1}-\rho_{s_2}=v'_{s_1}M({s_2})(u_{s_1}-u_{s_2})+(v_{s_1}-v_{s_2})'M(s_2)u_{s_2}=(v_{s_1}-v_{s_2})'M({s_2})(u_{s_1}-u_{s_2})\]
\begin{equation}\label{eqh0}
+\rho_{s_2}(v'_{s_2}(u_{s_1}-u_{s_2})+(v_{s_1}-v_{s_2})'u_{s_2})=(v_{s_1}-v_{s_2})'M({s_2})(u_{s_1}-u_{s_2})+\rho_{s_2}(v'_{s_2}u_{s_1}+v'_{s_1}u_{s_2}-2).
\end{equation}
Inequalities $(i)$ and $(ii)$ of Lemma \ref{l2} guarantee (a consequence of the Implicit function theorem) that the vectors $v_{s}=v(s,f)$ are differentiable functions of $s$ on the compact  $C_{s_0}\times \mathcal{K}$. Therefore, we can express their Taylor expansions as:
\begin{equation}\label{eqh1}
v_{s_1} = v_{s_2} + (V(s_2)- E_{v}(s_1))({s_1} - {s_2}) ,
\end{equation}
where $V(s_2)=\begin{Vmatrix}V_{kl}(s_2)\end{Vmatrix}_{k,l\in \{1,...,d\}}:V_{kl}(s_2)=\frac{\partial v_k (s_2)}{\partial s_l}$ and
$0\le E_{v}(s_1)\le V(s_2), [E_v(s_1)]_{k,l} \to 0, s_1\to s_2.$
Note that $V_{kl}(s)$ must be bounded on the compact $C_{s_0}\times \mathcal{K}$. Multiplying the transposed equation \eqref{eqh1} by $u_{s_2}$, we obtain
\[v'_{s_1}u_{s_2}-1=\left((V(s_2)- E_{v}(s_1))({s_1} - {s_2})\right)'u_{s_2}.\]
Similarly, we can show that 
\[v'_{s_2}u_{s_1}-1=-\left((V(s_2)- E_{v}(s_1))({s_1} - {s_2})\right)'u_{s_1}.\]
From these two equations, we get
\begin{equation}\label{eqh2aq}
v'_{s_2}u_{s_1}+v'_{s_1}u_{s_2}-2 = \left((V(s_2)- E_{v}(s_1))({s_1} - {s_2})\right)'(u_{s_2}-u_{s_1}).
\end{equation}
The vectors $u_s$ are also differentiable on the compact $C_{s_0}\times \mathcal{K}$ and their derivatives are bounded, so \eqref{eqh2aq} implies
\begin{equation}\label{eqh2}
\lvert v'_{s_2}u_{s_1}+v'_{s_1}u_{s_2}-2 \rvert \le c_m^1 \lvert s_1-s_2 \rvert^2.
\end{equation}
Lemma 2 $(iii)$ also shows that
\begin{equation}\label{eqh3}
\lvert (v_{s_1}-v_{s_2})'M({s_2})(u_{s_1}-u_{s_2}) \rvert \le c_m^2 \lvert s_1-s_2 \rvert^2.
\end{equation} 
Relations \eqref{eqh0}, \eqref{eqh2} and \eqref{eqh3} yield  \eqref{eqh}.
\end{proof}

\begin{lemma}\label{l6}
\begin{equation}\label{eq65}
\theta_n=-n\lvert 1- \rho \rvert (1+o_n),
\end{equation}
\begin{equation}\label{eq20a}
\rho_{h^*}\le 1.
\end{equation}
If $\rho>1$, then
\begin{equation}\label{eq19}
\textbf{1}-\mu=\frac{(\rho-1)}{Q}u(\textbf{1}+o),
\end{equation}
\begin{equation}\label{eq20}
(1-\rho_{\mu})=(\rho-1)(1+o).
\end{equation}
\end{lemma}
\begin{proof}[Proof]
Since $\mu = f(\mu)=...=f_n(\mu)$, by \eqref{eq13} we get 
\begin{equation}\label{eq21}
\textbf{1}-\mu=\textbf{1}-f_n(\mu)=\rho^n \frac{v'(\textbf{1}-\mu)}{1+\pi_nQv'(\textbf{1}-\mu)}(u+o_n).
\end{equation}
Multiplying the left-hand side of \eqref{eq21} by $v_k$, summing from $k=1$ to $k=d$, and taking into consideration \eqref{eq2}, we get that 
\[v'(\textbf{1}-\mu)=\rho^n \frac{v'(\textbf{1}-\mu)}{1+\pi_nQv'(\textbf{1}-\mu)}(1+o_n).\]
Solving this equation for $v'(\textbf{1}-\mu)$ gives
\[v'(\textbf{1}-\mu)=-\frac{1-\rho^n(1+o_n)}{Q\pi_n}.\]
If $\rho^n \le \hat{c}$ as $\sigma_2 \to 0$, then $1-\rho^n(1+o_n)=(1-\rho^n)(1+o_n)$; if $\rho^n \to \infty$ as $\sigma_2 \to 0$, then $\rho^n/\pi^n=(\rho-1)(1+o_n)$. So, in both cases, we have
\[v'(\textbf{1}-\mu)=(\rho-1)Q^{-1}(1+o_n),\]
which yields \eqref{eq19} by \eqref{eq13a} ($o_n=o$, since $\textbf{1}-\mu$ doesn't directly depend on $n$).

Since $\textbf{1} - \mu =o$ by the first equation in \eqref{eq21} and \eqref{eq13aa}, and $\rho_s$ are differentiable functions of $s$ and $f$ on the compact, it must be the case that
\[\rho = \rho_{\mu} +o.\]
Therefore, in order to prove \eqref{eq20}, it is sufficient to show that $\rho_{\mu}<1$. 

Suppose the opposite - $\rho_{\mu}>1$.
Let $u_{\mu}$ be the right eigenvector of  $\rho_{\mu},$ which we know is positive. 
Take a small $\alpha >0$, $k \in \{1,2,...,d\}$, then

\begin{gather} 
\nonumber f_{1,k}(\mu-\alpha u_{\mu})=f_{1,k}(\mu) - \alpha  \sum_{l=1}^d M^k_l(\mu) u_{\mu,l} + o(\alpha) \\
 =\mu_k - \alpha \rho_{\mu}u_{\mu,k}+o(\alpha)=\mu_k - \alpha u_{\mu,l} - (\rho_{\mu}-1)\alpha u_{\mu,l}(1+o(1))<\mu_k - \alpha u_{\mu,l}.\label{eq21a}
\end{gather}
Relation \eqref{eq21a} implies that the map $f=f(s)$ maps the set $\textbf{0} \le s \le \mu-\alpha u_{\mu}$ to itself. Therefore, by Brauer's theorem, there exists a fixed point $s=f(s)$ on $\textbf{0} \le s \le \mu-\alpha u_{\mu}$. But $\mu$ itself is the smallest fixed point on $\textbf{0} \le s \le \textbf{1}$, which contradicts our assumption that $\rho_{\mu}>1$ and proves \eqref{eq20}.

Relation \eqref{eq65} follows immediately from \eqref{eq20} and the Taylor expansion for the logarithm.
Consider inequalities
\begin{equation}\label{eqin}
h_n(s) < h^*(s) \le h^*(1-)=\mu, \text{ } n \ge 1.
\end{equation}
From these inequalities, it is evident that the elements of the matrix $M(\mu)$ are greater than the elements of $M(h^*)$. Therefore, for any natural $n$:
\begin{equation}\label{eq01}
M^n(\mu)\ge M^n(h^*).
\end{equation}
Suppose the opposite, $\rho_{h^*} > 1$. 
Then, by Lemma \ref{l3}, $M^n(h^*)/\rho^n_{h^*}$ converges to $u_{h^*}v'_{h^*},$  which implies that 
$M^n(h^*) \sim \rho^n_{h^*} u_{h^*}v'_{h^*}$. By the same argument $M^n(\mu) \sim \rho^n_{\mu} u_{\mu}v'_{\mu}$.
Since eigenvectors $u_s$ and $v_s$ are uniformly bounded and $\rho_{\mu}<1$, it follows that
\[M^n(\mu)<M^n(h^*),\]
which contradicts to \eqref{eq01}.
\end{proof}

\begin{lemma}\label{l7}
For all $s$, such that $\textbf{0} < s_0 \le s < \textbf{1}$
\begin{equation}\label{eqin0}
v'_{{\mu}}\mu(\textbf{1}-s)= v'_{\mu}(\mu-h^*)o(s),
\end{equation}
\begin{equation}\label{eqin000}
 v'_{\mu}(\mu-h^*)=o(s),
\end{equation}
\begin{equation}\label{eqin2a}
v_{h^*}'(t_{n}(s)-h^*) =o_n(s).
\end{equation}
\end{lemma} 
\begin{proof}[Proof]
Expansion \eqref{eq16} gives 
\begin{equation}\label{ea22}
h^*=sf(h^*)=s  \left(\mu - M_{\mu}(\mu-h^*)+\hat{q}[\mu,\mu-h^*]\right).
\end{equation}
Let $s_m=\min_{k\in \{1,2,...,d\}}(s_k)$. Then, after subtracting $sh^*$ from the previous equation and premultiplying it by $v'_{\mu}$, we get
\begin{gather}
\nonumber v'_{\mu}\left(s(\mu-M_{\mu}(\mu-h^*))-sh^*\right)=v'_{\mu}s((I-M_{\mu})(\mu-h^*)) \\
\ge s_mv'_{\mu}((I-M_{\mu})(\mu-h^*)) = s_m(1-\rho_{\mu})v'_{\mu}(\mu-h^*) \ge 0.\label{eqin1}
\end{gather}
Relations \eqref{eqin} and \eqref{eqin1} imply
\[v'_{\mu}s\hat{q}[\mu,\mu-h^*] \le v'_{\mu}h^*(\textbf{1}-s) \le v_{\mu}'\mu(\textbf{1}-s).
\]
Using \eqref{ea22},  previous inequality,  and \eqref{eq20}, we derive
\[
v'_{{\mu}}(\mu-h^*)=v'_{{\mu}}\mu(\textbf{1}-s)+v'_{{\mu}}s\left(M_{{\mu}}(\mu-h^*)-\hat{q}[\mu,\mu-h^*]\right)\]
\[=v'_{{\mu}}\mu(\textbf{1}-s)+v'_{{\mu}}(\mu-h^*)(1+o_n(s)).\]
From those equations follows \eqref{eqin0}.

Using \eqref{eq16}, we have
\[\mu-h^*=\mu-sf(h^*)=\mu(\textbf{1}-s)+s\left(M_{\mu}-E_{h^*}\right)(\mu-h^*),\]
from which follows
\[v'_{\mu}(\mu-h^*)\le v'_{\mu}(\textbf{1}-s)+v'_{\mu}M_{\mu}(\mu-h^*)= v'_{\mu}(\textbf{1}-s)+\rho_{\mu}v'_{\mu}(\mu-h^*).\]
Thus,
\[v'_{\mu}(\mu-h^*) \le \frac{v'_{\mu}(\textbf{1}-s)}{1-\rho_{\mu}}.\]
Last inequality, part $(2)$ of Remark \ref{r1}, and Lemma \ref{l2} $(vii)$ prove \eqref{eqin000}.

Since 
\begin{equation}\label{eqin1252}
h^*\le t_n(s),\mu\le\textbf{1},
\end{equation}
 $\mu-h^*=o(s)$ by \eqref{eqin000}, and $\textbf{1}-\mu=o$ by \eqref{eq19}, we obtain \eqref{eqin2a}.

\end{proof}
For the rest of the paper, we will make an assumption the authors believe to be true:
\begin{equation}\label{eqin2}
v_{h^*}'(h^*-h_{n}(s)) = o_n(s).
\end{equation}
We note that for the compact subset $\mathcal{K}_{b_1,p} \subset \mathcal{K}, p,b_1>0$, where $p_{m}(\textbf{0})>p, b^k_{lm}>b_1, k,l,m \in \{1,...,d\}$, it is easy to show that \eqref{eqin2} holds.
\begin{lemma} \label{l8}
Let $\sigma \to 0$, then
\begin{equation}\label{eq18u}
h^*-h_{n}(s)=v'_{h^*}(h^*-h_{n}(s))u_{h^*}(\textbf{1}+o_n(s));
\end{equation}
\begin{equation}\label{eq18t}
t_{n}(s)-h^*=v'_{h^*}(t_{n}(s)-h^*)u_{h^*}(\textbf{1}+o_n(s));
\end{equation}
\begin{equation}\label{eq18m}
\mu-h^*=v'_{\mu}(\mu-h^*)u_{\mu}(\textbf{1}+o(s));
\end{equation}
\end{lemma}
\begin{proof}[Proof]
Choose sequence $f^n\in\mathcal{K},$ for which $\rho(f^n)\to1$. Denote $\rho_{h^*}=\rho_{h^*}(f^n)=\rho_{h^*}(n,s).$
Let $A(n,m,s)=\rho^{-1}_{h^*(n,s)}E(h_m(s))$, $P(n,s)=\rho^{-1}_{h^*(n,s)}M(h^*)$.

Repeatedly using expansion \eqref{eq15}, we obtain
\[h^*-h_{n}(s)=s(f(h^*)-f(h_{n-1}(s)))=s\left(M(h^*)-E(h_{n-1}(s))\right)(h^*-h_{n-1}(s))\]
\[=\ldots=\prod_{m=1}^{n-1}s\left(M(h^*)-E(h_m(s))\right)(h^*-h_0(s))=\rho^{n}_{h^*(n,s)}\prod_{m=1}^{n-1}s\left(P(n,s)-A(n,m,s))\right)h^*.\]
Let $s_{min}=\min_{k\in\{1,...,d\}}s_k,$ $s_{max}=\max_{k\in\{1,...,d\}}s_k,$. Clearly,
\begin{gather}
\nonumber  s^n_{min}\rho^{n}_{h^*(n,s)}\prod_{m=1}^{n-1}\left(P(n,s)-A(n,m,s))\right)h^*\le \rho^{n}_{h^*(n,s)}\prod_{m=1}^{n-1}s\left(P(n,s)-A(n,m,s))\right)h^* \\
\nonumber \le s^n_{max}\rho^{n}_{h^*(n,s)}\prod_{m=1}^{n-1}\left(P(n,s)-A(n,m,s))\right)h^*.
\end{gather}
From Remark \ref{r1}, it is clear that $s^n_{min}=s^n_{max}=1+o_n(s)$. Thus, the previous inequalities give
 \begin{equation}\label{eq18mms}
h^*-h_{n}(s)=\rho^{n}_{h^*(n,s)}\prod_{m=1}^{n-1}\left(P(n,s)-A(n,m,s))\right)h^*(\textbf{1}+o_n(s)).
\end{equation}
By Lemma \ref{l3}, $P(n,s)$ satisfy condition $(i)$ of Lemma \ref{l1}.

Then the monotonicity of $h_m(s)$ gives
\begin{equation}\label{eq18a1}
m_1<m_2 \implies A(n,m_2,s)\le A(n,m_1,s).
\end{equation}
Define $p_{n,m}(s)=\eta^{-2}\max_{l,j}A_{lj}(n,m,s)$. For fixed $m$, $A(n,n-m,s)$ tends to zero by  \eqref{eqin2} and \eqref{eq15A}, hence $p_{n,m}(s)$ also tends to zero. Also, from \eqref{eq18a1}, $p_{n,m}(s)$ is non-increasing in $m$ for fixed $n$ and $s$. Lemma \ref{l2} $(vi)$ and $(vii)$ imply that 
\[A(n,m,s) \le p_{n,m}(s)R(n,s),\]
which means that condition $(ii)$ of Lemma \ref{l1} is satisfied, and the following convergence is true
\[\prod_{m=1}^{n-1}\left(P(n,s)-A(n,m,s))\right)h^*/v'_{h^*}\prod_{m=1}^{n-1}\left(P(n,s)-A(n,m,s))\right)h^*= u_{h^*}(\textbf{1}+o_n(s)).\]

Since sequence $f^n$ is arbitrary, the last result and \eqref{eq18mms} prove \eqref{eq18u}.

Now, let's proceed to prove \eqref{eq18m}. Choose sequence $f^n\in\mathcal{K},$ for which $\rho(f^n)\to1$. Denote $\mu(f^n)=\mu_n, h^*(f^n)=h^*_n$.
Since $\mu=f(\mu)$ and $h^*=sf(h^*)$, using \eqref{eq15} repeatedly, we derive
\begin{gather}\label{eq18a2}
\mu_n-h^*_n=\mu_n(\textbf{1}-s)+s(M_{{\mu}_n}-E(h^*_n))(\mu_n-h^*_n)\\
=\ldots
\nonumber =\sum_{k=0}^{n-1}\left(s(M_{{\mu}_n}-E(h^*_n))\right)^k\mu_n(\textbf{1}-s)+\left(s(M_{{\mu}_n}-E(h^*_n))\right)^n(\mu_n-h^*_n).
\end{gather}
Similarly to \eqref{eq18mms}, it is clear that
\begin{gather}
\nonumber  \sum_{k=0}^{n-1}s^k_{min}\left(M_{{\mu}_n}-E(h^*_n)\right)^k(\mu_n-h^*_n)+s^n_{min}\left(M_{{\mu}_n}-E(h^*_n)\right)^n(\mu_n-h^*_n)\le \mu_n-h^*_n \\
\nonumber \le \sum_{k=0}^{n-1}s^k_{max}\left(M_{{\mu}_n}-E(h^*_n)\right)^k\mu_n(\textbf{1}-s)+s^n_{max}\left(M_{{\mu}_n}-E(h^*_n)\right)^n(\mu_n-h^*_n).
\end{gather}
From Remark \ref{r1}, we see that $s^k_{min}=s^k_{max}=1+o_n(s), k \in \{0,1,..,n\}$, so the previous inequality give
 \begin{gather}
 \mu_n-h^*_n =\Sigma(\textbf{1}+o_n(s)) \label{eq18mmds}\\
\nonumber =\left(\sum_{k=0}^{n-1}\left(M_{{\mu}_n}-E(h^*_n))\right)^k\mu_n(\textbf{1}-s)+\left((M_{{\mu}_n}-E(h^*_n))\right)^n(\mu_n-h^*_n)\right)(\textbf{1}+o_n(s)).
\end{gather}
Let $P(n)=\rho^{-1}_{{\mu}_n}M_{{\mu}_n}$, $A(n,m,s)=A(n,s)=\rho^{-1}_{\mu_n}E(h^*_n), B(n,s)=P(n)-A(n,s)$.

Define $p_{n}(s)=\eta^{-2}\max_{l,j}A_{lj}(n,s)$. It is null by \eqref{eq15A} and \eqref{eqin000}. Lemma \ref{l2} $(vi)$ and $(vii)$ show that 
\begin{equation}\label{eq18uo}
A(n,s) \le p_n(s)R(n).
\end{equation}

We can rewrite the expression for $\Sigma$ as follows
\begin{gather}
 \nonumber \Sigma=\sum_{k=0}^NB^k(n,s)\rho^{k}_{\mu}\mu_n(\textbf{1}-s) \label{eq18b11b}\\
+\sum_{k=N+1}^{n-1}B^k(n,s)\rho^{k}_{\mu_n}\mu_n(\textbf{1}-s)+B^n(n,s)\rho^{n}_{\mu_n}(\mu_n-h^*_n)=\Sigma_1+\Sigma_2, \label{eq18u1}
\end{gather}
where $\Sigma_2=\sum_{k=N+1}^{n-1}B^k(n,s)\rho^{k}_{\mu_n}\mu_n(\textbf{1}-s)+B^n(n,s)\rho^{n}_{\mu_n}(\mu_n-h^*_n).$

Lemma \ref{l3} guarantees the existence of a null sequence $\delta_m$, such that 
\[(1-\delta_m)R(n)\le P^m (n)\le (1+\delta_m)R(n), \]
$R(n)=u_{{\mu}_n}v'_{{\mu}_n}$. Therefore, for $k>N$ 
\begin{equation}\label{eq19d1}
\prod_{j=k-N+1}^{k}B(n,s)\le P^N(n)  \le  R(n)(1+\delta_N).
\end{equation}
Inequality \eqref{eq18uo} and the fact that $PR=RP=R$ shows that
\begin{gather}
 \nonumber \prod_{j=k-N+1}^{k}B(n,s)\ge \prod_{j=k-N+1}^{k}(P(n)-p_n(s) R(n))\\
\ge P^N(n)-Np_n(s) R(n)\ge (1-\delta_N-Np_n(s))R(n). \label{eq19d2}
\end{gather}
Inequalities \eqref{eq19d1} and \eqref{eq19d2}  show 
\begin{gather}
\nonumber (1-\delta_N-Np_n(s))R(n)\left(\sum_{k=N+1}^{n-1}B^{k-N}(n,s)\rho^{k}_{\mu_n}\mu_n(\textbf{1}-s)+B^{n-N}(n,s)\rho^{n}_{\mu_n}(\mu_n-h^*_n)\right)\\
\le \Sigma_2 \le (1+\delta_N)R(n)\left(\sum_{k=N+1}^{n-1}B^{k-N}(n,s)\rho^{k}_{\mu_n}\mu_n(\textbf{1}-s)+B^{n-N}(n,s)\rho^{n}_{\mu_n}(\mu_n-h^*_n)\right).\label{eq19d3}
\end{gather}
The same inequalities hold for $v'_{\mu_n}\Sigma_2$.
Relations \eqref{eqin0} and \eqref{eq18mmds} show that
\begin{equation}\label{eq19d33}
\frac{v'_{\mu_n}\mu_n(\textbf{1}-s)}{v'_{\mu_n}\Sigma} =o_n(s).
\end{equation}
It is easy to see that for fixed $N$, we can find positive constants $c^1_N$ and $c^2_N$, such that
\begin{equation}\label{eq19si}
 c^1_N\min_{k\in \{1,2,...,d\}}(1-s_k) \le \Sigma_1 \le c^2_N\max_{k\in \{1,2,...,d\}}(1-s_k).
\end{equation}
Then, \eqref{eq19si}, \eqref{eq19d33}, \eqref{eq18u1}, and part $(1)$ of Remark \ref{r1} show that both
\begin{equation}\label{eq19sig}
\frac{\Sigma_1}{v'_{\mu_n}\Sigma_2}= o_n(s),\frac{v'_{\mu_n}\Sigma_1}{v'_{\mu_n}\Sigma_2} =o_n(s).
\end{equation}

Dividing equation \eqref{eq18u1} by itself multiplied by $v'_{\mu_n}$, we get
\[ \frac{\Sigma}{v'_{\mu_n}\Sigma}=\frac{\Sigma_1/v'_{\mu_n}\Sigma_2+\Sigma_2/v'_{\mu_n}\Sigma_2}{1+v'_{\mu_n}\Sigma_1/v'_{\mu_n}\Sigma_2}.\]
Note that $\frac{R_{\mu_n}x}{v'_{\mu_n}R_{\mu_n}x} = u_{\mu_n}$ regardless of the vector $x$, due to the normalization $v'_{\mu_n}u_{\mu_n}=1$. Therefore, using \eqref{eq19d3} and the previous equation, we have 
\[\frac{(1-\delta_N-Np_n(s))u_{\mu_n}}{(1+\delta_N)(1+ v'_{\mu_n}\Sigma_1/v'_{\mu_n}\Sigma_2)}\le \frac{\Sigma}{v'_{\mu}\Sigma}
\le \frac{(1+\delta_N)u_{\mu_n}}{(1-\delta_N-Np_n(s))}+\frac{\Sigma_1}{v'_{\mu_n}\Sigma_2}.
\]
Inequalities $(vi)$ and $(vii)$ from Lemma \ref{l2} give the following
\[\frac{\Sigma_1}{v'_{\mu_n}\Sigma_2} \le \frac{v'_{\mu_n}\Sigma_1}{\eta^2 v'_{\mu_n}\Sigma_2} u_{\mu_n}.\]
Thus, 
\[\Big\lvert \left[ \frac{\Sigma}{v'_{\mu}\Sigma}-u_{\mu_n} \right]_k \Big\rvert \le  \frac{2\delta_N+Np_n(s)+(1+\delta_N)v'_{\mu_n}\Sigma_1/\eta^2 v'_{\mu_n}\Sigma_2}{1-\delta_N-Np_n(s)}.\]
Fix $\epsilon >0$. Choose $N$ large enough, so that $\delta_N<\frac{\epsilon}{6}$. Then, we can select $n^1_{\epsilon}$ and $s^1_{\epsilon}$, so that $Np_n(s)<\frac{\epsilon}{6}$ for all $n>n_{\epsilon}$ and all $s>s^1_{\epsilon}$. Furthermore, based on \eqref{eq19sig}, we can find $n^2_{\epsilon}>n^1_{\epsilon}$ and $s^2_{\epsilon}>s^1_{\epsilon}$, so that $(1+\delta_N)v'_{\mu_n}\Sigma_1/\eta^2 v'_{\mu_n}\Sigma_2$ is less than $\frac{\epsilon}{6}$ for all $n>n^2_{\epsilon}$ and $s>s^2_{\epsilon}$. Consequently,
\[\Big\lvert \frac{\Sigma}{v'_{\mu}\Sigma}-u_{\mu_n} \Big\rvert < \frac{2\epsilon}{3-\epsilon}  < \epsilon, \epsilon<1.\]
Since sequence $f^n$ is arbitrary, the last relation, combined with \eqref{eq18mmds} gives \eqref{eq18m}.

The proof of \eqref{eq18t} follows a similar approach to that of \eqref{eq18u}, and thus, we will not repeat it.
\end{proof}

Define 
\[
a=a(s)=\rho_{h^*}v_{h^*}'su_{h^*}, Q_{h^*}=v_{h^*}'sq[h^*,u_{h^*}].\]
Additionally, let 
\[ R_1=R_1(s,n)=Q_{h^*}(1-a^n), R_2=R_2(s,n)=\frac{1-a}{s-h^*}, R=\frac{R_1}{R_2}.\]
\begin{lemma}\label{l9}
 Let $\sigma \to 0$, then
\begin{equation}\label{eq26a}
h^*-h_{n}(s)=\frac{(1-a)a^n}{R_1}u_{h^*}(\textbf{1}+o_n(s));
\end{equation}
\begin{equation}\label{eq26}
 t_{n}(s)-h_{n}(s)=\frac{(1-a)a^n\left(1+(1+R)o_n(s)\right)}{R_1(1-R(1+o_n(s)))}u_{h^*}(\textbf{1}+o_n(s)).
\end{equation}
\begin{equation}\label{eq26b}
g_{n}(s)=\frac{Q_{h^*}(1-a)^2a^n}{R^2_1}u_{h^*}(\textbf{1}+o_n(s));
\end{equation}
\end{lemma} 
\begin{proof}[Proof]
Relations \eqref{eq16}, \eqref{eq16a}, and \eqref{eqin2} give
\begin{equation}\label{eq27}
h^*-h_{n+1}(s)=sM_{h^*}(h^*-h_{n}(s))-(v'_{h^*}(h^*-h_{n}(s)))^2sq[h^*,u_{h^*}](1+o_n(s)).
\end{equation}
After multiplying \eqref{eq27} from the left-hand side by $v'_{h^*}$,  we obtain
\begin{equation}\label{eq27a}
v'_{h^*}(h^*-h_{n+1}(s))=\rho_{h^*}v'_{h^*}(h^*-h_{n}(s))\frac{v'_{h^*}sM_{h^*}(h^*-h_{n}(s))}{\rho_{h^*}v'_{h^*}(h^*-h_{n}(s))}-(v'_{h^*}(h^*-h_{n}(s)))^2Q_{h^*}(1+o_n(s)).
\end{equation}
Bearing in mind the normalization $v'_{h^*}u_{h^*}=1$ and using \eqref{eq18u} in the numerator and the denominator, we get
\begin{gather}
\nonumber \rho_{h^*}\frac{v'_{h^*}sM_{h^*}(h^*-h_{n}(s))}{\rho_{h^*}v'_{h^*}(h^*-h_{n}(s))}=a\left(1-\left(1-\frac{(v_{h^*}s)'u_{h^*}(\textbf{1}+o_n(s))}{(v_{h^*}s)'u_{h^*}v'_{h^*}u_{h^*}(\textbf{1}+o_n(s))}\right)\right)\\
\nonumber  =a\left(1-\frac{((v_{h^*}s)'u_{h^*}v'_{h^*}-v'_{h^*}+v'_{h^*}-(v_{h^*}s)')u_{h^*}o_n(s)}{(v_{h^*}s)'u_{h^*}v'_{h^*}u_{h^*}(\textbf{1}+o_n(s))}\right)\\  
=a\left(1-\frac{((v_{h^*}(\textbf{1}-s))'-(v_{h^*}(\textbf{1}-s))'u_{h^*}v'_{h^*})u_{h^*}o_n(s)}{(v_{h^*}s)'u_{h^*}v'_{h^*}u_{h^*}(\textbf{1}+o_n(s))}\right). \label{eq27b}
\end{gather}
Let $x_n=v'_{h^*}(h^*-h_{n}(s))$, $b_n(s)=1-\frac{((v_{h^*}(\textbf{1}-s))'-(v_{h^*}(\textbf{1}-s))'u_{h^*}v'_{h^*})u_{h^*}o_n(s)}{(v_{h^*}s)'u_{h^*}v'_{h^*}u_{h^*}(\textbf{1}+o_n(s))}$. Then, using the previous equation and \eqref{eq27a}, we obtain the equation
\begin{equation}\label{eq29}
x_{n+1}=ab_n(s)x_n-Q_{h^*}x^2_n(1+o_n(s)).
\end{equation}
The same reasoning applied to $t_{n}(s)-h^*$ (using \eqref{eq16}, \eqref{eq16a}, \eqref{eqin2a}, and \eqref{eq18t}), gives the equation
\begin{equation}\label{eq30}
x_{n+1}=ab_n(s)x_n+Q_{h^*}x^2_n(1+o_n(s)),
\end{equation}
where $x_{n}=v'_{h^*}(t_{n}(s)-h^*).$
Note that 
\begin{equation}\label{eq31}
a=1+o(s), Q_{h^*}=Q(h^*)+o(s).
\end{equation}
Now proceed as in \cite[ p.~442]{k4}.
Let $y_n=1/x_n$. Then \eqref {eq29} and \eqref {eq30}  give
\begin{equation}\label{eq32}
y_{n}=ab_n(s)y_{n+1} \mp Q_{h^*}\frac{x_n}{x_{n+1}}(1+o_n(s)).
\end{equation}
The second part of Remark \ref{r1} and Lemma \ref{r2} $ (v)$ show that
\begin{equation}\label{eq32b}
\prod_{l=1}^kb_l(s)=1+o, k \in \{1,...,n\}.
\end{equation}
Inequality \eqref{eqin2} shows that $\frac{x^2_n}{x_{n+1}}=o_n.$ Thus, equations \eqref {eq29}-\eqref {eq31} and \eqref {eq32b} demonstrate that
\[\frac{x_n}{x_{n+1}}=\frac{1+o_n(s)}{ab_n(s)}=1+o_n(s),\]
so we can rewrite \eqref{eq32} as
\[y_{n+1}=a^{-1}b^{-1}_n(s)y_{n}\pm Q_{h^*}(1+o_n(s)).\]
Iteration of this equation yields 
\[y_{n+1}=\pm Q_{h^*}a^{-n} \sum_{k=0}^na^{k}\prod_{l=0}^kb^{-1}_{n-l}(1+o_k(s))
+a^{-n-1}\prod_{l=0}^{n}b^{-1}_{n-l}(s)y_0.
\]
Hence, by \eqref {eq32b} and Lemma \ref{l4}, we obtain
\[y_{n+1}=\pm Q_{h^*}a^{-n} \sum_{k=0}^na^{k}(1+o_n(s))+a^{-n-1}y_0(1+o),\]
and consequently
\begin{equation}\label{eq32a1}
x_{n}=\left(a^{-n}x_0^{-1}\pm Q_{h^*}\frac{(a^{-n}-1)}{(a-1)}\right)^{-1}(1+o_n(s)).
\end{equation}
For equation \eqref{eq29} $x_0=h^*$, and for equation \eqref{eq30} $x_0=s-h^*$. Using this and rearranging \eqref{eq32a1} 
we get that 
\begin{equation}\label{eq33}
v'_{h^*}(h^*-h_{n}(s))=\frac{(1-a)a^n}{R_1}(1+o_n(s)), \text{ } v'_{h^*}(t_{n}(s)-h^*)=\frac{(1-a)a^n}{R_2(1-R(1+o_n(s)))}.
\end{equation}
These relations, along with \eqref{eq18u} and \eqref{eq18t}, prove \eqref{eq26a} and \eqref{eq26}.

Using \eqref{eq8a}, similarly to \eqref{eq18u}, one can show that
\begin{equation}\label{eq33a}
g_n(s) = v'_{h^*}(h_{n+1}(s)-h_{n}(s))u_{h^*}(\textbf{1}+o_n(s)).
\end{equation}
Also, using \eqref{eq8a}, \eqref{eq27a}, and \eqref{eq27b} we obtain 
\begin{equation}\label{eq33b}
v'_{u^*}(h_{n+1}(s)-h_{n}(s)) = (1-a+a(1-b_n(s)))v'_{h^*}({h^*}-h_{n}(s))+Q_{h^*}(v'_{h^*}({h^*}-h_{n}(s)))^2(1+o_n(s)).
\end{equation}
Relation \eqref{eq26a} yields
\[(1-a)v'_{h^*}({h^*}-h_{n}(s))+Q_{h^*}(v'_{h^*}({h^*}-h_{n}(s)))^2(1+o_n(s))=\frac{Q_{h^*}(1-a)^2a^n}{R^2_1}(1+o_n(s)+a^no_n(s)).\]
Note that $1-a=1-\rho_{h^*}+\rho_{h^*}(1-v'_{h^*}su_{h^*})=1-\rho_{h^*}+\rho_{h^*}(v'_{h^*}(\textbf{1}-s)u_{h^*})$, so that $v'_{h^*}(\textbf{1}-s)u_{h^*}$ has a rate of at least the same order as $1-a$.
This, along with \eqref{eq26a}, implies
\[ a(1-b_n(s))v'_{h^*}({h^*}-h_{n}(s))=\frac{(1-a)^2a^{n+1}}{R_1}o_n(s).\]
Since $a<1$  by \eqref{eq20a}, the last two equations and \eqref{eq33b} give
\[v'_{h^*}(h_{n+1}(s)-h_{n}(s)) = \frac{Q_{h^*}(1-a)^2a^n}{R^2_1}(1+o_n(s)),\]
which, combined with \eqref{eq33a}, yields \eqref{eq26b}.
\end{proof}

Define
\[W=W(s,\rho_{\mu})=\frac{4Q(\mu)v'_{\mu}(\textbf{1}-s)}{(1-\rho_{\mu})^2}, V=V(s,\rho_{\mu})=\sqrt{1+W}.\]
\begin{lemma} \label{l10}
 Let $\sigma_2 \to 0$, then
\begin{equation}\label{eq35}
\mu-h^*=\frac{(1-\rho_{\mu})(V-1)}{2Q(\mu)}u_{\mu}(\textbf{1}+o(s)).
\end{equation}
\end{lemma} 
\begin{proof}[Proof]
Representation \eqref{eq16}-\eqref{eq16a} for $h^*$ and \eqref{eqin000} give 
\[h^*=s(\mu-M(\mu)(\mu-h^*)+q[\mu,\mu-h^*](1+o(s))).\]
Let $x=v'_{\mu}(\mu-h^*)$. After subtracting this equation from $\mu$ and multiplying it  by $v'_{\mu}$ and using \eqref{eq18m}, we obtain the equation
\begin{equation}\label{eq36}
x=v'_{\mu}\mu(\textbf{1}-s)+\rho_\mu\frac{v'_{\mu}sM(\mu)(\mu-h^*)}{\rho_\mu v_{\mu}'(\mu-h^*)}x-x^2(v_{\mu}s)'q[\mu,u_{\mu}](1+o(s)).
\end{equation}
Then, \eqref{eq18m} yields $\frac{v'_{\mu}sM(\mu)(\mu-h^*)}{\rho_\mu v_{\mu}'(\mu-h^*)}=(v_{\mu}s)'u_{\mu}\left(1-\frac{-(v_{\mu}(\textbf{1}-s))'u_{\mu}v'_{\mu}+(v_{\mu}(\textbf{1}-s))')u_{\mu}o(s)}{(v_{\mu}s)'u_{\mu}v'_{\mu}u_{\mu}(\textbf{1}+o(s))}\right)$ by the same arguments as in \eqref{eq27b}. This, along with \eqref{eq19d33}, implies
\begin{equation}\label{eq37a}
\frac{v'_{\mu}sM(\mu)(\mu-h^*)}{\rho_\mu v_{\mu}'(\mu-h^*)}x=(v_{\mu}s)'u_{\mu}x+x^2o(s).
\end{equation}
It is also evident that due to the definition of $Q(s)$, we have
\begin{equation}\label{eq37}
(v_{\mu}s)'q[\mu,u_{\mu}]=Q(\mu)(1+o(s)).
\end{equation}
We can express $1-\rho_\mu(v_{\mu}s)'u_{\mu}$ as $1-\rho_\mu+\rho_\mu(1-(v_{\mu}s)'u_{\mu})$. Then, using \eqref{eq37a} and \eqref{eq37}, we can rewrite equation \eqref{eq36} as
\begin{gather*}
Q(\mu)x^2(1+o(s))+(1-\rho_\mu)x\\
-v_{\mu}'(\textbf{1}-s)\left(1-\frac{v'_{\mu}(\textbf{1}-\mu)(\textbf{1}-s)+\rho_\mu(1-(v_{\mu}s)'u_{\mu} )x}{v_{\mu}'(\textbf{1}-s)}\right)=0.
\end{gather*}
It follows straightforwardly that the fraction in the last term is $o(s)$. Therefore we get the equation 
\[
Q(\mu)x^2(1+o(s))+(1-\rho_\mu)x-v_{\mu}'(\textbf{1}-s)(1+o(s))=0.
\]
This equation has one positive solution
\begin{equation}\label{eq37a7}
x=\frac{(1-\rho_{\mu})((1+W(1+o(s))^{\frac{1}{2}}-1)}{2Q(\mu)}(1+o(s)),
\end{equation}
In view of the representations $1+W(1+o(s))=(1+W)(1+Wo(s)/(1+W))$ and $(1+x)^{\frac{1}{2}}=1+\frac{x}{2}(1+o(1))$, we have 
\[(1+W(1+o(s)))^{\frac{1}{2}}-1=V-1+\frac{W}{2(1+W)}o(s).\]
Now, noticing that $\frac{W}{2(1+W)(V-1)}\le 1+o(s)$ and combining the previous equation with \eqref{eq37a7} and \eqref{eq18m}, we obtain \eqref{eq35}.
\end{proof}

\begin{lemma} \label{l11} If $-\theta_n(V-1) \le \hat{c}$ and $-\theta_n V \ge \hat{c}$,  then
\begin{equation}\label{eq39}
1-a=(1-\rho_{\mu})V(1+o(s));
\end{equation}
\begin{equation}\label{eq40}
a^n=e^{\theta_n V}(1+o_n(s));
\end{equation}
\begin{equation}\label{eq40a}
1-a^n=(1-e^{\theta_n V})(1+o_n(s)).
\end{equation}
\end{lemma} 
\begin{proof}[Proof]
Using \eqref{eqin000}, we can rewrite representation \eqref{eq18} as
\[M(h^*)=M(\mu)+q(h^*,\mu)(1+o(s)).\]
Multiplying this equation by $v'_{\mu}$ from the left-hand side and by $u_{\mu}$ from the right-hand side, and using \eqref{eq35}, we obtain
\begin{equation}\label{eq41}
v'_{\mu}M(h^*)u_{\mu}=\rho_{\mu}-(1-\rho_{\mu})(V-1)(1+o(s)).
\end{equation}
Next, we multiply equation \eqref{eq41} by $(v_{h^*}s)'u_{h^*}$ and rearrange to get
\[(v_{h^*}s)'u_{h^*}v'_{\mu}M(h^*)u_{\mu}=\rho_{\mu}-(1-\rho_{\mu})(V-1)\left(1+\rho_{\mu}\frac{(1-(v_{h^*}s)'u_{h^*})}{(1-\rho_{\mu})(V-1)}+o(s)\right).\]
Straightforward but tedious calculations, analogous to those in the one-dimensional case (\cite[Lemma~19]{k4}), show that  $\rho_{\mu}\frac{(1-(v_{h^*}s)'u_{h^*})}{(1-\rho_{\mu})(V-1)}=o(s)$ Therefore, we have
\begin{equation}\label{eq42}
(v_{h^*}s)'u_{h^*}v'_{\mu}M(h^*)u_{\mu}=\rho_{\mu}-(1-\rho_{\mu})(V-1)(1+o(s)).
\end{equation}
Relations \eqref{eqh} and \eqref{eq35} show that 
\begin{equation}\label{eq43}
v'_{\mu}M(h^*)u_{\mu}-\rho_{h^*}=O\left((1-\rho_{\mu})^2(V-1)^2\right )
\end{equation}
and thus,
\begin{equation}\label{eq45}
1-a=1-(v_{h^*}s)'u_{h^*}v'_{\mu}M(h^*)u_{\mu}+(v_{h^*}s)'u_{h^*}v'_{\mu}M(h^*)u_{\mu}-a=(1-\rho_{\mu})V(1+o(s)),
\end{equation}
proving \eqref{eq39}.
Relations \eqref{eq43} and \eqref{eq45} imply that we can replace the left-hand side of \eqref{eq42} by $a$. We can express $
\rho_{\mu}-1$ as $\ln\rho_{\mu}(1+o)$ and get 
\[a=\rho_{\mu}(1+(V-1)\ln\rho_{\mu}(1+o(s))).\]
Given the first condition of the lemma, $-(V-1)\ln\rho_{\mu}=o_n(s)$, so we can use the approximation,
\[\ln(1+(V-1)\ln\rho_{\mu}(1+o(s)))=(V-1)\ln\rho_{\mu}(1+o_n(s)).\]
Thus,
\[a^n=\rho_{\mu}^ne^{\theta_n (V-1)}(1+o_n(s))=e^{\theta_n V}(1+o_n(s)),\]
which proves \eqref{eq40}.
Relation \eqref{eq40a} follows from the previous equation and the second condition of the lemma.
\end{proof}

\begin{lemma} \label{l12} If $\rho \gtrless 1$  and $\sigma \to 0$, then,  then
\begin{equation}\label{eqs}
s-h^*=\frac{(1-\rho_{\mu})(V \pm1)}{2Q}u(\textbf{1}+o(s)).
\end{equation}
\end{lemma} 
\begin{proof}[Proof]
Equation \eqref{eq9} yields 
\[s-h^*=s(\textbf{1}-f(h^*))=s(\textbf{1}-\mu+\mu-h^*+h^*-f(h^*))=s(\textbf{1}-\mu+\mu-h^*+f(h^*)(\textbf{1}-s)).\]
Relation \eqref{eqin000} allow us to rewrite this equation as
\begin{equation}\label{eq46}
s-h^*=s(\textbf{1}-\mu+(\mu-h^*)(\textbf{1}+o(s))).
\end{equation}
Relation $\textbf{1}-\mu=o$ and the continuity of $Q(s)$ and $u_s$ yield
\begin{equation}\label{eq47}
Q(\mu)=Q\cdot(1+o), u_{\mu}=u(\textbf{1}+o).
\end{equation}
If $\rho < 1$, then $\mu=\textbf{1}$, and \eqref{eqs} follows immediately from \eqref{eq35} and \eqref{eq47}.
If $\rho > 1$, then using \eqref{eq19}, \eqref{eq35}, as well as  \eqref{eq20} and \eqref{eq47}, we can rewrite \eqref{eq46} as 
\[s-h^*=\left(\frac{1-\rho_{\mu}}{Q}(\textbf{1}+o)+\frac{(1-\rho_{\mu})(V-1)}{2Q}(\textbf{1}+o(s))\right)u\]
\[=\left(\frac{(1-\rho_{\mu})((V+1)\textbf{1}+o)+(1-\rho_{\mu})(V-1)o(s)}{2Q}\right)u=\frac{(1-\rho_{\mu})(V+1)}{2Q}u(\textbf{1}+o(s)),\]
finishing the proof.
\end{proof}

\begin{lemma} \label{l13} In the conditions of Lemma \ref{l11}
\begin{equation}\label{eqf2}
g_{n}(s)=\frac{V^2e^{V\theta_n}(1-\rho_{\mu})^2}{Q(1-e^{V\theta_n})^2}u(\textbf{1}+o_n(s)).
\end{equation}
\begin{equation}\label{eqf1}
h^*-u_{n}(s)=\frac{Ve^{V\theta_n}(1-\rho_{\mu})}{Q(1-e^{V\theta_n})}u(\textbf{1}+o_n(s));
\end{equation}
 If also $W \ge \hat{c}$ in the case when $\rho>1$,  then for $\rho \gtrless 1$
\begin{equation}\label{eqf3}
t_{n}(s)-h_{n}(s)=\frac{2V^2e^{V\theta_n}(1-\rho_{\mu})}{Q(1-e^{V\theta_n})(V\mp1+(V\pm1)e^{V\theta_n})}u(\textbf{1}+o_n(s)).
\end{equation}
\end{lemma} 
\begin{proof}[Proof]
Similarly to \eqref{eq47}, with the help of \eqref{eqin000} and \eqref{eq19}, we  can show that
\begin{equation}\label{eq48}
Q(h^*)=Q\cdot(1+o(s)), u_{h^*}=u(\textbf{1}+o(s)).
\end{equation}
Now, representations \eqref{eqf1} and \eqref{eqf2} follow immediately from \eqref{eq26a} and \eqref{eq26b}, using \eqref{eq39}-\eqref{eq40a} as well as \eqref{eq48}.

Relations \eqref{eq39}-\eqref{eq40a} and \eqref{eq48} yield 
\begin{equation}\label{eq50}
\frac{(1-a)a^n}{R_1}=\frac{V(1-\rho_{\mu})e^{V\theta_n}}{Q(1-e^{V\theta_n})}(1+o_n(s)),
\end{equation}
and
\[R=\frac{(1-e^{V\theta_n})(V\pm1)}{2V}(1+o_n(s)).\]
Hence,
\begin{equation}\label{eq49}
R \le \hat{c}
\end{equation}
and
\[1-R(1+o_n(s))=\frac{2V-(1-e^{V\theta_n})(V\pm1)(1+o_n(s))}{2V}\]
\[=\frac{V\mp1+(V\pm1)e^{V\theta_n}+(1-e^{V\theta_n})(V\pm1)o_n(s)}{2V}.\]
Consider the quantity
\[D=\frac{(1-e^{V\theta_n})(V\pm1)}{V\mp1+(V\pm1)e^{V\theta_n}}.\]
Clearly, $V\ge1$ and $e^{V\theta_n}\le1$. If $\rho<1$, then it is easy to see that $D<(V-1)/(V+1)<\hat{c}$. In case $\rho>1$, we have $D<(V+1)/(V-1)$. Condition $W \ge \hat{c}$ in this case guarantees that $D<\hat{c}$. 

This shows that
\begin{equation}\label{eq51}
1-R(1+o_n(s))=\frac{V\mp1+(V\pm1)e^{V\theta_n}}{2V}(1+o_n(s)).
\end{equation}
Combining \eqref{eq26}, \eqref{eq50}-\eqref{eq51} yields \eqref{eqf3}.
\end{proof}

\begin{lemma}\label{l14}
If $i=2$, the conditions of Lemma \ref{l11} are satisfied, and $W \to 0$ as $\sigma \to 0$, then
\begin{equation}\label{eqr1}
g_{n}(s)=\frac{\rho_{\mu}^ne^{W\theta_n/2}(1-\rho_{\mu})^2}{Q}u(\textbf{1}+o_n(s)).
\end{equation}
If $\rho > 1$, then
\begin{equation}\label{eqr2}
h^*-h_{n}(s)=\frac{1-\rho_{\mu}}{Q}uo_n(s).
\end{equation}
If $\rho < 1$, then
\begin{equation}\label{eqr3}
t_{n}(s)-h_{n}(s)=\frac{\rho^{n}e^{W\theta_n/2}(1-\rho)}{Q}u(\textbf{1}+o_n(s)).
\end{equation}
\end{lemma}
\begin{proof}[Proof]
Using the definition of $W$, we observe that
\begin{equation}\label{eqr4}
W=o_n(s).
\end{equation}
Furthermore, by the expansion $(1+x)^{\frac{1}{2}}=1+\frac{x}{2}(1+o(1))$, we have
\begin{equation}\label{eqr5}
V-1=\frac{W}{2}(1+o_n(s)).
\end{equation}
From here and the first condition of Lemma \ref{l11}, we deduce
\begin{equation}\label{eqr6}
e^{V\theta_n}=\rho_{\mu}^ne^{W\theta_n/2}(1+o_n(s)).
\end{equation}
Now, considering that $\rho_{\mu}^n=o_n(s)$ in our case, equations \eqref{eqr1}, \eqref{eqr2}, and \eqref{eqr3} follow from \eqref{eqf2}, \eqref{eqf1}, and \eqref{eqf3} respectively, using relations \eqref{eqr4} - \eqref{eqr6}.
\end{proof}

\begin{lemma} \label{l0} If $\rho\ne1$, then
\[
 E\left(Y_n\middle|Z_0=e_m\right)=\frac{1-\rho^{n+1}}{1-\rho}u_mv(1+o_n).
\]
\end{lemma}
\begin{proof}[Proof]
It is easy to see that $M^n$ is the matrix of the first moments of $f_n(s)$. Then, by Lemma \ref{l3}, we have
 \[E\left(Z_n\middle|Z_0=e_m\right)=u_mv\rho^n(\textbf{1}+o_n).\] 
Lemma \ref{l4} yields
\begin{gather*}
E\left(Y_n\middle|Z_0=e_m\right)=\sum_{j=0}^{n}E\left(Z_j\middle|Z_0=e_m\right)u_mv\rho^j(\textbf{1}+o_j)\\
 =\sum_{j=0}^{n}E\left(Z_j\middle|Z_0=e_m\right)u_mv\rho^j(\textbf{1}+o_n)=u_mv\frac{1-\rho^{n+1}}{1-\rho}(\textbf{1}+o_n),
\end{gather*}
finishing the proof.
\end{proof}
\begin{lemma}\label{l15} Let $i=2$, $\rho>1$, $nv'(1-s)/(\rho-1) \to 0$ as $\sigma \to 0$. Then
\begin{equation}\label{eqv}
v'(\kappa-t_n(s))=\left(\frac{Q}{\rho-1}+\frac{\rho-1}{\rho^nv'(\textbf{1}-s)}\right)^{-1}(1+o_n(s)),
\end{equation}
where 
\[\kappa=\kappa(s)=\textbf{1}+u\frac{\rho (vs)'u-1}{2(vs)'uQ}\left(\sqrt{1+\frac{4(vs)'uQv'(\textbf{1}-s)}{(\rho (vs)'u-1)^2}}-1\right).\]
\end{lemma} 
\begin{proof}[Proof]
Expansion \eqref{eq16} allow us to write 
\begin{equation}\label{eqvv}
t_{n+1}(s)=s(f(\textbf{1})-M(\textbf{1}-t_{n}(s))+q[\textbf{1}-t_{n}(s)](1+o_n(s))).
\end{equation}
Lemma \ref{l0} shows that $\lim_{s\to \textbf{1}}(1-t_{n,j}(s))/(1 -s_k) =(1-\rho^n)/(1-\rho)u_jv_k(1+o_n),$ $\textbf{  }j,k \in \{1,...,d\}$. Given that $(1-\rho^n)/(1-\rho) \to \infty, \sigma_1 \to 0$, we have
\begin{equation}\label{eqv1r}
v'(\textbf{1}-s)=v'(\textbf{1}-t_{n}(s))o_n(s).
\end{equation}
Using \eqref{eq18u}-\eqref{eq18m} as well as \eqref{eqin000}-\eqref{eqin2}, one can show that
\[\textbf{1}-t_{n}(s)=v'(\textbf{1}-t_{n}(s))u(\textbf{1}+o_n(s)).\]
Therefore, after subtracting \eqref{eqvv} from $\textbf{1}$ and multiplying the left-hand side by the left eigenvector, we obtain
\begin{gather}
\nonumber v'(\textbf{1}-t_{n+1}(s))=\rho v'(\textbf{1}-t_{n}(s)) \frac{v'sM(\textbf{1}-t_{n}(s))}{\rho v'(\textbf{1}-t_{n}(s))}\\
-b_1(v'(\textbf{1}-t_{n}(s)))^2(1+o_n(s))+v'(\textbf{1}-s),\label{eqv1}
\end{gather}
where $b_1=v'sq[u]$. Using the same arguments as in \eqref{eq27b}, we obtain
\[
\frac{v'sM(\textbf{1}-t_{n}(s))}{\rho v'(\textbf{1}-t_{n}(s))}=(vs)'u\left(1-\frac{-(v(\mathbf{1}-s)'uv'+(v(\mathbf{1}-s))')uo_n(s)}{(vs)'uv'u(\mathbf{1}+o_n(s))}\right).
\]
This equation and \eqref{eqv1r} show that 
\[v'(\textbf{1}-t_{n}(s)) \frac{v'sM(\textbf{1}-t_{n}(s))}{v'(\textbf{1}-t_{n}(s))}=(vs)'uv'(\textbf{1}-t_{n}(s))+(v'(\textbf{1}-t_{n}(s)))^2o_n(s),\]
so \eqref{eqv1} can be rewritten as 
\begin{equation}\label{eqv1a}
v'(\textbf{1}-t_{n+1}(s))=v'(\textbf{1}-t_{n}(s))\rho (vs)'u-b_1(v'(\textbf{1}-t_{n}(s)))^2(1+o_n(s))+v'(\textbf{1}-s).
\end{equation}
Also, note that
\begin{equation}\label{eqv2a}
v'(\kappa-\textbf{1})=\frac{\rho (vs)'u-1}{2(vs)'uQ}\left(\sqrt{1+\frac{4(vs)'uQv'(\textbf{1}-s)}{(\rho (vs)'u-1)^2}}-1\right)\end{equation}
is a solution of the equation 
\begin{equation}\label{eqv2}
b_1x^2+(\rho (vs)'u-1)x-v'(\textbf{1}-s)=0.
\end{equation}
Since $v'(\textbf{1}-t_{n}(s))=v'(\textbf{1}-\kappa)+v'(\kappa-t_{n}(s))$, combining \eqref{eqv1a} and \eqref{eqv2} gives
\[v'(\kappa-t_{n+1}(s))=\rho (vs)'uv'(\kappa-t_{n}(s))\]
\[+b_1\left((v'(\textbf{1}-\kappa))^2-(v'(\textbf{1}-t_{n}(s)))^2\right)(1+o_n(s))\]
\[=((\rho (vs)'u+2b_1v'(\kappa-\textbf{1}))v'(\kappa-t_{n}(s))-b_1(v'(\kappa-t_{n}(s)))^2(1+o_n(s)).\]
Let $x_n=v'(\kappa-t_{n}(s)),$ $a_1=\rho (vs)'u+2b_1v'(\kappa-\textbf{1})$. Then we get the equation
\begin{equation}\label{eqv3}
x_{n+1}=a_1x_n+b_1x_n^2(1+o_n(s)).
\end{equation}
The third condition of the lemma yield 
\begin{equation}\label{eqv3b}
\rho (vs)'u-1=(\rho-1)(1+o_n(s)), \frac{v'(\textbf{1}-s)}{(1-\rho)^2}=o_n(s).
\end{equation}
From those relations and \eqref{eqv2a}, we see that 
\begin{equation}\label{eqv3a}
v'(\kappa-\textbf{1})=\frac{v'(\textbf{1}-s)}{\rho-1}(1+o_n(s))
\end{equation}
and
\begin{equation}\label{eqv3d}
v'(\kappa-\textbf{1})=(\rho-1)o_n(s).
\end{equation}
Using the definition of $a_1$, the first relation in \eqref{eqv3b} and \eqref{eqv3d}, we obtain the equation
\begin{equation}\label{eqv4}
a_1-1=(\rho-1)(1+o_n(s)).
\end{equation}
Since $\kappa -t_n(s)= \kappa -\textbf{1}+\textbf{1}-s+s-h^*+h^*-t_n(s)$, relations \eqref{eqv3d}, \eqref{eqs}, and \eqref{eqin2a} yield
\begin{equation}\label{eqv5}
v'(\kappa-t_n(s))=o_n(s).
\end{equation}
Given \eqref{eqv4} and \eqref{eqv5}, solving recursively as in Lemma \ref{l9}, we find the solution to \eqref{eqv3}.
\begin{equation}\label{eqv6}
x_n=\left(b_1 a_1^{-n}\sum_{k=0}^na_1^k(1+o_n(s))+\frac{a_1^{-n-1}}{x_0}\right)^{-1}.
\end{equation}
From \eqref{eqv3a} and \eqref{eqv3d}, we see that $v'(\textbf{1}-s)=(1-\rho)o_n(s)$, and hence
\begin{equation}\label{eqv7}
x_0=v'(\kappa-s)=\frac{v'(\textbf{1}-s)}{\rho-1}(1+o_n(s)).
\end{equation}
Using the same logic as during the derivation of \eqref{eq40} in Lemma \ref{l12}, we conclude that
\[a_1^n=(\rho (vs)'u)^n\exp\left\{ \frac{Qnv'(\textbf{1}-s)}{\rho-1}(1+o_n(s))\right\}.\]
Using the third condition of the lemma, we further deduce from the last equation that
\[a_1^n=\rho^n(1+o_n(s))\]
and consequently
\[\sum_{k=0}^na_1^k=\frac{\rho^n}{\rho-1}(1+o_n(s),\]
which, combined with \eqref{eqv7} and \eqref{eqv6}, yields the result.

\end{proof}

\section{Main results}
\subsection{Processes without immigration}
Define
\[I_{+}(x)=e^x, I_{-}(x)=1, h_{\pm}(x)=\frac{(1-e^x)(e^{\mp x}-3)-2xI_{\pm}(x)}{(1-e^x)^3};\]
\[d_{\pm}(y,x)=\sqrt{1+\frac{4y}{(1-x)^2h_{\pm}(\ln{x})}}.\]

\begin{proposition}\label{p1}
\begin{equation}\label{eqm0}
P\left(Z_n>0\middle|Z_0=e_m\right)=
\begin{cases}
      1)\frac{(1-\rho_{\mu})I_{\mp}(\theta_n)}{Q(1-\rho_{\mu}^n)}u_m(1+o_n),i=1, \rho \gtrless 1;\\
      2)\frac{(1-\rho)\rho^n}{Q}u_m(1+o_n), i=2, \rho < 1;\\
      3)\frac{1-\rho}{Q}u_m(1+o_n), i=2, \rho > 1.
    \end{cases}
\end{equation}
\end{proposition}
\begin{proof}[Proof]
Definition \eqref{eq4} allow us to write 
\[\pi_n=\frac{1-\rho^n}{1-\rho}(1+o).\]
From \eqref{eq14}, we get 
\begin{equation}\label{eqm4}
P\left(Z_n>0\middle|Z_0=e_m\right)=\frac{\rho^{n-1}}{{Q}\sum_{j=0}^{n-1}\rho^{j}}u_m(1+o_n).
\end{equation}
Consider the case $i=1$. If $\rho < 1$, then the first equation of \eqref{eqm0} follows immediately from \eqref{eqm4}. If $\rho > 1$, it is clear from \eqref{eq20} that $\rho^k=\rho_{\mu}^{-k}(1+o_k)$.
Therefore, using Lemma \ref{l4}, we get 
\[P\left(Z_n>0\middle|Z_0=e_m\right)=\frac{\rho_{\mu}^{-n}}{{Q}\sum_{j=0}^{n-1}\rho_{\mu}^{-j}}u_m(1+o_n),\]
which again yields the first equation of \eqref{eqm0} after expanding the sum on the right-hand side.

If $i=2$, then in the case $\rho < 1$, we have $\rho^n=o_n$ and in the case $\rho > 1$, \text{ }  $\rho^n \to \infty$. Hence, the second
and the third equations in \eqref{eqm0} follow from \eqref{eqm4}. 
\end{proof}
Define 
\[\lim_{\sigma \to 0}\lvert n(1-\rho) \rvert = -\ln{r} < \infty;  T=\sum_{k=1}^d t_k;\]
\begin{equation}\label{eqm1}
m_n^1=
\begin{cases}
      1)\frac{Q(1-\rho_{\mu}^n)^2 h_{\pm}(\theta_n)}{(1-\rho_{\mu})^2}v,i=1, \rho \gtrless 1,r<1;\\
      2)\frac{2Qn^2}{3}v,i=1,r=1;\\
      3)\frac{2Qn}{1-\rho}v, i=2, \rho \nearrow 1;\\
      4)\frac{Q\rho^n}{(1-\rho)^2}v, i=2, \rho \searrow 1.
    \end{cases}
\end{equation}
\begin{theorem}\label{t3}
Let $s_n=\left(e^\frac{-t_1}{m^1_{n,1}},\ldots,e^\frac{-t_d}{m^1_{n,d}}\right), t_k \ge0, k \in \{1,...,d\}$. Then for any $j \in \{1,...,d\}$
\[\lim_{\sigma_1\to 0}E\left(s_n^{Y_{n}}\middle|Z_n>0,Z_0=e_j\right)\]
\begin{equation}\label{eq58}
=\begin{cases}
      1)\frac{2(1-r)r^{d_{\pm}( T,r)}(d_{\pm}(T,r))^2}{r^{(1\mp1)/2}(1-r^{d_{\pm}( T,r)})(d_{\pm}(T,r) \mp 1 + (d_{\pm}( T,r) \pm 1)r^{d_{\pm}( T,r)})}, i=1,  r < 1;\\
      2)\frac{\sqrt{6T}}{\text{sh}\sqrt{6T}},i=1, r =1;\\
      3) e^{-T},i=2, \rho \nearrow 1;\\
      4)\frac{1}{1+T},i=2, \rho \searrow 1,
         \end{cases}
\end{equation}
where the function $d_{+}(T,r)$ corresponds to the case $\rho \searrow 1$, and $d_{-}(T,r)$ corresponds to the case $\rho \nearrow 1$.
\end{theorem}
\begin{proof}[Proof]
It follows from \eqref{eqm1} that $t/m^2_n=o_n$, hence,
\begin{equation}\label{eq63}
1-s_{n,j}=\frac{t_j}{m^1_{n,j}}(1+t_jo_n).
\end{equation}
Let 
\[\Phi^1_{n,j}(t)=E\left(s_n^{Y_{n}}\middle|Z_n>0,Z_0=e_j\right),\]
then
\begin{equation}\label{eq62}
\Phi^1_{n,j}(t)=\frac{t_{n,j}\left(s_n\right)-h_{n,j}\left(s_n\right)}{P(Z_n>0,Z_0=e_j)}.
\end{equation}
Consider the case where $i=1$ and $r<1$. From the first equation of \eqref{eqm1}, \eqref{eq63}, and \eqref{eq47}, it follows that $W = \frac{4T(1+To_n)}{(1-\rho^n_{\mu})^2 h_{\pm}(\theta_n)}$, and thus
\[V=\sqrt{1+4T\frac{1+To_n}{(1-\rho^n_{\mu})^2 h_{\pm}(\theta_n)}}=d_{\pm}(T(1+To_n),e^{\theta_n}).\]
Now consider the case $i=1$, $r=1$. The second equation in \eqref{eqm1}, combined with \eqref{eq65}, \eqref{eq63}, and \eqref{eq47}, shows that $W = \frac{6T(1+To_n)}{n^2(1-\rho_{\mu})^2} \to \infty$ and $V = -\frac{\sqrt{6T}}{\theta_n}(1+To_n)$.

Verifying the conditions of Lemma \ref{l11} is straightforward. In both cases, \eqref{eq62}, \eqref{eqf3}, and the first equation of \eqref{eqm0} show that
\[ \Phi^1_{n,j}(t)=\frac{2e^{\theta_n V}(1-e^{{\theta_n}})V^2}{I_{\mp}(\theta_n)(1-e^{\theta_n V})(V\mp1+(V\pm1)e^{\theta_n V})}(1+o_n).\]
Since $e^{\theta_n} \to r$ as $\sigma_1 \to \infty$  by \eqref{eq65}, the previous equation yields the first two cases in \eqref{eq58}.

Consider now the case where $i=2$ and $\rho \nearrow 1$. Relations \eqref{eq63}, \eqref{eq65}, and the third relation in \eqref{eqm1} show that $W = \frac{-2T(1+To_n)}{\theta_n} \to 0$ as $\sigma_1 \to 0$. These relations also imply that the conditions of Lemma \ref{l14} are satisfied. Thus, \eqref{eq62}, \eqref{eqr3}, and the third equation in \eqref{eqm0} yield $\Phi^m_n(t)=e^{\theta_n W/2}(1+o_n(s_n))=e^{-T}(1+o_n(s_n))$, which completes the proof of the third case in \eqref{eq58}.

Now, let's consider the case where $i=2, \rho \searrow 1$.
We can represent $t_n(s_n)-g_n(s_n)$ as
\[t_n(s_n)-h_n(s_n)=(\kappa-\textbf{1})+(\textbf{1}-s_n)+(s_n-h^*)+(h^*-h_n(s_n))-(\kappa-t_n(s_n)).\]
Relations \eqref{eqr2}, \eqref{eqv3d}, as well as \eqref{eq63} with the fourth equation in \eqref{eqm1}, yield 
\[v'\left((\kappa-\textbf{1})+(\textbf{1}-s_n)+(h^*-h_n(s_n))\right)=(\rho-1)o_n(s_n).\]
Also, \eqref{eqs} and \eqref{eq20} show that $v'(s_n-h^*)=\frac{\rho-1}{Q}(1+o(s)).$ Our choice of $m^1_n$ and \eqref{eq63} show that the conditions of Lemma \ref{l15} are satisfied. Relation \eqref{eqv} and the fourth equation in \eqref{eqm1} show that $v'(\kappa-t_n(s_n))=\frac{(\rho-1)T}{Q(1+T)}(1+o_n(s_n)).$ As a result, $v'(t_n(s_n)-h_n(s_n))=\frac{\rho-1}{Q(1+T)}(1+To_n(s_n))$. This, \eqref{eqm0}, \eqref{eq18u}, \eqref{eq18t} and \eqref{eq62} yields the fourth case in \eqref{eq58}.
\end{proof}
\begin{remark}
Let $G^1_{\pm}(x,r)=\frac{2(1-r)r^{d_{\pm}( x,r)}(d_{\pm}(x,r))^2}{r^{(1\mp1)/2}(1-r^{d_{\pm}(x,r)})(d_{\pm}(x,r) \mp 1 + (d_{\pm}( x,r) \pm 1)r^{d_{\pm}( x,r)})}$. Then, it is easy to check that $\lim_{r \to 1}G^1_{\pm}(x,r)=\frac{\sqrt{6x}}{\text{sh}\sqrt{6x}}$, $\lim_{r \to 0}G^1_{+}(x,r)=\frac{1}{1+{x}}$, $\lim_{r \to 0}G^1_{-}(x,r)=e^{-x}$.
\end{remark}
\begin{proposition}\label{p2}
\begin{equation}\label{eqm1b}
P\left(N=n\middle|Z_0=e_m\right)=
\begin{cases}
      1)\frac{\rho_{\mu}^n(1-\rho_{\mu})^2}{Q(1-\rho_{\mu}^n)^2}u_m(1+o_n),i=1, \rho \gtrless 1;\\
      2)\frac{\rho_{\mu}^n(1-\rho_{\mu})^2}{Q}u_m(1+o_n), i=2.
         \end{cases}
\end{equation}
\end{proposition}
\begin{proof}[Proof]
From \eqref{eq8a}, we conclude that 
\[P\left(N=n\middle|Z_0=e_m\right)=g_{n,m}(\textbf{1}).\]
Then $(1)$  follows from here and \eqref{eqf2}. For $(2)$, it remains to note that $\rho_{\mu}^n = o_n$.
\end{proof}
Define 
\[h(x)=\frac{x(1+e^x)-2(1-e^x)}{(e^x-1)^3}; d(y,x)=\sqrt{1+\frac{2y}{(1-x)^2h(\ln{x})}};\]
\begin{equation}\label{eqm1aa}
m_n^2=
\begin{cases}
      1)\frac{2Q(1-\rho_{\mu}^n)^2 h(\theta_n)}{(1-\rho_{\mu})^2}v,i=1, \rho \gtrless 1, r < 1;\\
      2)\frac{Qn^2}{3}v,i=1, r= 1;\\
      3)\frac{2Qn}{1-\rho_{\mu}}v, i=2.
    \end{cases}
\end{equation}

\begin{theorem} \label{t4}
Let $s_n=\left(e^\frac{-t_1}{m^2_{n,1}},\ldots,e^\frac{-t_d}{m^2_{n,d}}\right), t_k \ge0, k \in \{1,...,d\}$. Then for any $j \in \{1,...,d\}$
\begin{equation}\label{eq66a}
\lim_{\sigma_1\to 0} E\left(s_n^{Y_{n}}\middle|N=n,Z_0=e_j\right)=\begin{cases}
      1)\frac{(1-r)^2r^{d( T,r)}(d(T,r))^2}{r(1-r^{d( T,r)})^2}, i=1, r \ne 1;\\
      2)\frac{3T}{\text{sh}^2\sqrt{3T}},i=1, r =1;\\
      3)e^{-T}, i=2.
         \end{cases}
\end{equation}
\end{theorem}
\begin{proof}[Proof]
It follows from \eqref{eqm1} that $t/m^2_n=o_n$, and thus
\begin{equation}\label{eq63a}
1-s_{n,j}=\frac{t_j}{m^2_{n,j}}(1+t_jo_n).
\end{equation}
Let 
\[\Phi^2_{n,j}(t)=E\left(s_n^{Y_{n}}\middle|N=n,Z_0=e_j\right),\]
then
\begin{equation}\label{eq62a}
\Phi^2_{n,j}(t)=\frac{g_{n,j}\left(s_n\right)}{P(N=n,Z_0=e_j)}.
\end{equation}
Consider the case $i=1$ and $r<1$. From the first equation of \eqref{eqm1aa}, \eqref{eq63a}, and \eqref{eq47}, it follows that $W = \frac{2T(1+To_n)}{(1-\rho^n_{\mu})^2 h(\theta_n)}$, and thus
\[V=\sqrt{1+2T\frac{1+To_n}{(1-\rho^n_{\mu})^2 h(\theta_n)}}=d(T(1+To_n),e^{\theta_n}).\]
Consider the case $i=1$, $r=1$. 
The second equation in \eqref{eqm1aa}, \eqref{eq63a}, \eqref{eq47}, and \eqref{eq65} show that $W = \frac{12T(1+To_n)}{\theta_n^2} \to \infty$ and $V = -2\frac{\sqrt{3T}}{\theta_n}(1+To_n)$.

It is easy to verify that the conditions of Lemma \ref{l11} are satisfied. Then, in both cases, \eqref{eq62a}, \eqref{eqf2}, and the first equation of \eqref{eqm1b} show that
\[ \Phi^2_{n,j}(t)=\frac{e^{\theta_n V}(1-e^{{\theta_n}})^2V^2}{(1-e^{\theta_n V})^2e^{\theta_n }}(1+o_n).\]
Since $e^{\theta_n} \to r$ as $\sigma_1 \to \infty$ by \eqref{eq65}, the previous equation yields the first two cases in \eqref{eq66a}.

Now, consider the case $i=2$. Relations \eqref{eq63a}, \eqref{eq65}, and the second relation in \eqref{eqm1aa} show that $W = \frac{-2T(1+To_n)}{\theta_n} \to 0$ as $\sigma_1 \to 0$. These relations also imply that the conditions of Lemma \ref{l14} are satisfied. Hence, \eqref{eq62a}, \eqref{eqr1}, and the second relation in \eqref{eqm1b} yield $\Phi^m_n(t)=e^{\theta_n W/2}(1+o_n(s))=e^{-T}(1+o_n(s))$, which completes the proof of \eqref{eq66a}.
\end{proof}
\begin{remark}
Let $G^2(x,r)=\frac{(1-r)^2r^{d( x,r)}(d(x,r))^2}{r(1-r^{d( x,r)})^2}$. Then, it is easy to check that $\lim_{r \to 1}G^2(x,r)=\frac{3x}{\text{sh}^2\sqrt{3x}}$, $\lim_{r \to 0}G^2(x,r)=e^{-x}$.
\end{remark}

\subsection{Processes with immigration}

Define
\begin{equation}\label{eqimm1}
m_n^3=
\begin{cases}
      1)\frac{\rho_{\mu}^{\mp n}-1\mp n(1-\rho_{\mu})}{(1-\rho_{\mu})^2}v,i=1, \rho \gtrless 1, r < 1;\\
      2)\frac{n^2}{2}v, i=1, \rho  \gtrless 1, r=1;\\
      3)\frac{n}{1-\rho}v, i=2, \rho \nearrow 1;\\
      4)\frac{\rho^n}{(1-\rho)^2}v, i=2, \rho  \searrow  1;
    \end{cases}
\end{equation}

\[z_{\pm}(x,y)={\sqrt{1+\frac{4Qy}{x^{\mp1}-1 \pm \ln{x}}}}; \text{ } w_{\pm}(x,y) = \frac{z_{\pm}(x,y) \mp1}{z_{\pm}(x,y)\pm1};\]
\[\Psi_1(r,T)=\begin{cases}
      1)\left(r^{\frac{\mp1}{2}}\frac{r^{z_{\pm}(r,T)/2}+w_{\pm}(r,T)r^{-z_{\pm}(r,T)/2}}{1+w_{\pm}(r,T)}\right)^{-\frac{\lambda'u}{Q}}, i=1,r < 1;\\
      2)\left(\text{ch}\sqrt{2QT}\right)^{-\frac{\lambda'u}{Q}},i=1, r =1;\\
     3)e^{-\lambda'uT},  i=2,\rho \nearrow 1 ;\\
     4) \frac{1}{(1+{QT})^{\frac{\lambda'u}{Q}}}, i=2,\rho \searrow 1,
\end{cases}.\]

\begin{theorem}\label{t5}
Let $s_n=\left(e^\frac{-t_1}{m^3_{n,1}},\ldots,e^\frac{-t_d}{m^3_{n,d}}\right), t_k \ge0, k \in \{1,...,d\}$. Then
\begin{equation}\label{eq58i}
\lim_{\sigma_1\to 0}\lvert E\left(s_n^{Y^0_{n}}\right)-\Psi_1(r,T)\rvert=0.
\end{equation}
where the functions $z_{+}(T,r),w_{+}(T,r)$ correspond to the case $i=1,\rho \searrow 1$, and $z_{-}(T,r),w_{-}(T,r)$ correspond to the case $i=1, \rho \nearrow 1$.
\end{theorem}
\begin{proof}[Proof]
First consider the case $i=1, r \ne1$.

From \eqref{eq10}, we have
\[E\left(e^{-\sum_{k=1}^d t_kY^0_{n,k}/m^3_{n,k}}\right) = \phi_n(s_n).\] 
Note that since $\frac{t}{m^3_n}=o_n$, we have
\begin{equation}\label{eq63im}
{1}-s_n=\frac{t_j}{m^3_{n,j}}(1+t_jo_n).
\end{equation}
Then, similarly to the one-dimensional case (Pakes \cite[p.~287]{k9}),
\[
\ln{\phi_n(s_n)}= - \sum_{k=0}^{n-1}\left(1-B(t_k(s_n)) \right) -  \sum_{k=0}^{n-1}R_{n,k}(s_n),\]
where
\[ 0 \le R_{n,k}(s_n) \le \frac{\left(1-B(t_k(s_n)) \right)^2}{B(t_k(s_n))} \le \frac{\left(1-B(t_k(s_n)) \right) \left(1-B(t(s_n)) \right)}{B(t(s_n))}.\]
Using expansion \eqref{eq15im}, we obtain
\begin{equation}\label{eq65im}
\ln{\phi_n(s_n)}= - \sum_{k=0}^{n-1} \lambda'(\textbf{1}-t_k(s_n)) + \sum_{k=0}^{n-1} D'[t_k(s_n)](\textbf{1}-t_k(s_n))  -  \sum_{k=0}^{n-1}R_{n,k}(s_n),
\end{equation}
where the second and the third sums tend to zero if first sum is bounded.

Note that by \eqref{eqs}, \eqref{eq18t} and \eqref{eq48}, we have 
\begin{equation}\label{eq66im}
s_n-t_k(s_n)=v'(s_n-t_k(s_n))u(\textbf{1}+o_n(s_n)).
\end{equation}
From \eqref{eq63im}, \eqref{eqimm1}, and Lemma \ref{l2} $(viii)$, we also get that
\begin{equation}\label{eq67im}
n\lambda'(\textbf{1}-s_n) \to 0 \text{ as } \sigma \to 0.
\end{equation}
Relations \eqref{eq65im}-\eqref{eq67im} imply
\begin{equation}\label{eq68im}
\ln{\phi_n(s_n)}= - \lambda'u \sum_{k=0}^{n-1}v'(s_n-t_k(s_n))(1+o_k(s_n))  + \sum_{k=0}^{n-1} D'[t_k(s_n)](\textbf{1}-t_k(s_n))  -  \sum_{k=0}^{n-1}R_{n,k}(s_n).
\end{equation}
From the first equation of \eqref{eqimm1}, \eqref{eq63im}, and \eqref{eq47}, it follows that $W = \frac{4QT(1+To_n)}{\rho_{\mu}^{\mp n}-1{\mp}n(1-\rho_{\mu})}$, and thus
\begin{equation}\label{eq69im}
V=\sqrt{1+4QT\frac{1+To_n}{\rho_{\mu}^{{\mp}n}-1\mp n(1-\rho_{\mu})}}=z_{\pm}(T(1+To_n),e^{\theta_n}).
\end{equation}
Expressing $v'(s_n-t_k(s_n))$ as $v'(s_n-h^*) +v'(h^*-t_k(s_n))$, from \eqref{eqs} we obtain $v'(s_n-h^*)=\frac{(1-\rho_{\mu})(V\pm1)}{2Q}(1+o(s_n)),$
and therefore
\begin{equation}\label{eq70im}
\sum_{k=0}^{n-1}v'(s_n-h^*)=n\frac{(1-\rho_{\mu})(V\pm1)}{2Q}(1+o(s_n)) \to -\frac{(z_{\pm}(T,r)\pm1)\ln{r}}{2Q}, \sigma \to 0,
\end{equation}
since $e^{\theta_n} \to r$ as $\sigma_1 \to \infty$ by \eqref{eq65}.

We are again under the conditions of  Lemma \ref{l11}. Combining the second equation in \eqref{eq33}, \eqref{eq39}, \eqref{eq40}, and  \eqref{eqs}, we obtain
\begin{equation}\label{eq70imm}
v'(t_{k}(s_n)-h^*(s_n))=\frac{V(V\pm1)e^{V\theta_k}(1-\rho_{\mu})}{Q(V\mp1+(V\pm1)e^{V\theta_k})}(1+o_k(s_n))
\end{equation}
Also note that $n(1-\rho_{\mu})=-\ln{r}(1+o_n)$. Using this, the previous relation, \eqref{eq69im}, and Lemma \ref{l4} give
\begin{gather}
\nonumber \sum_{k=0}^{n-1}v'(t_k(s_n)-h^*) \\
\nonumber =- \frac{1}{n}\sum_{k=0}^{n-1}\frac{\ln{r}z_{\pm}(T(1+To_n),e^{\theta})(z_{\pm}(T(1+To_n),e^{\theta})\pm1)r^{m/nz_{\pm}(T(1+To_n),e^{\theta})}(1+o_k(s_n))}{Q\left( z_{\mp}(T(1+To_n),e^{\theta})\pm1 \right)+r^{m/nz_{\pm}(T(1+To_n),e^{\theta})}(z_{\pm}(T(1+To_n),e^{\theta})\pm1)}\\
= - \frac{1}{n}\sum_{k=0}^{n-1}\frac{\ln{r}z_{\pm}(T,r)(z_{\pm}(T,r)\pm1)r^{m/nz_{\pm}(T,r)}}{Q\left( z_{\pm}(T,r)\mp1 +r^{m/nz_{\pm}(T,r)}(z_{\pm}(T,r)\pm1)\right)}(1+o_n) \label{eqfim}\\
\nonumber \underset {n \to \infty}{\longrightarrow} -\int_0^1\frac{\ln{r}z_{\pm}(T,r)(z_{\pm}(T,r)\pm1)r^{x z_{\pm}(T,r)}}{Q\left( z_{\pm}(T,r)\mp1 +r^{x z_{\pm}(T,r)}(z_{\pm}(T,r)\pm1)\right)}dx  = -\frac{1}{Q} \ln{\frac{w_{\pm}(T,r)+r^{z_{\pm}(T,r)}}{1+w_{\pm}(T,r)}}.
\end{gather}
Combining \eqref{eq70im}, \eqref{eqfim}, \eqref{eq68im} and Lemma \ref{l2} $(viii)$ yields the first case in \eqref{eq58i}.

Consider the case $i=1, r =1$. The second equation in \eqref{eqimm1}, along with relations \eqref{eq63im}, \eqref{eq47}, and \eqref{eq65} show that $W = \frac{8QT(1+To_n)}{\theta_n^2} \to \infty$ and $V=\sqrt{8QT(1+To_n)}/\theta_n$. Then, by \eqref{eqs}
\begin{equation}\label{eqfim1}
\sum_{k=0}^{n-1}v'(s_n-h^*)=\sqrt \frac{2T}{Q}(1+To_n)\underset {n \to \infty}{\longrightarrow}\sqrt \frac{2T}{Q}.
\end{equation}
Relations \eqref{eq70imm} and \eqref{eqs} yield
\begin{gather}
\nonumber \sum_{k=0}^{n-1}v'(t_k(s_n)-h^*)=\frac{1}{n}\sum_{k=0}^{n-1}\frac{\sqrt{8QT(1+To_n)}e^{-m/n\sqrt{8QT(1+To_n)}}(1+o_k(s_n))}{Q\left(1+e^{-m/n\sqrt{8QT(1+To_n)}}\right)}\\
=\frac{1}{n}\sum_{k=0}^{n-1}\frac{\sqrt{8QT}e^{-m/n\sqrt{8QT}}}{Q\left(1+e^{-m/n\sqrt{8QT}}\right)}(1+o_n(s_n))\underset {n \to \infty}{\longrightarrow} \frac{1}{Q}\int_0^1\frac{2\sqrt{2QT}e^{-2x\sqrt{2QT}}}{1+e^{-2x\sqrt{2QT}}}dx \label{eqfim2}\\
\nonumber =-\frac{1}{Q}\ln{\frac{1+e^{-2\sqrt{2QT}}}{2}}.
\end{gather}
Combining \eqref{eqfim1}, \eqref{eqfim2}, and \eqref{eq68im} yields the second case in \eqref{eq58i}.

Consider case $i=2,  \rho \nearrow 1$. The third equation in \eqref{eqimm1}, relations \eqref{eq63im}, \eqref{eq47}, and \eqref{eq65} show that $W = \frac{4QT(1+To_n)}{\theta_n} \to 0$ and by \eqref{eqr5}, 
\begin{equation}\label{eqfim3a}
V-1=\frac{2QT(1+To_n)}{\theta_n}.
\end{equation}
Then, using \eqref{eqfim3a} and \eqref{eqs}, we get
\begin{equation}\label{eqfim3}
\sum_{k=0}^{n-1}v'(s_n-h^*)=T(1+To_n)\underset {n \to \infty}{\longrightarrow}T.
\end{equation}
Using \eqref{eq65}, we can write
\[e^{V\theta_k}= e^{\theta_k}e^{(V-1)\theta_k} = \rho_{\mu}^ke^{\frac{2kQT}{n}(1+o_k)} \le c_1(T) \rho_{\mu}^k.\]
Using the last inequality, \eqref{eq70imm}, \eqref{eqfim3a}, and the relation $V=1+o_n$, we have
\begin{gather*}
\sum_{k=0}^{n-1}v'(t_k(s_n)-h^*)=\sum_{k=0}^{n-1}\frac{T e^{V\theta_k}(1+To_k(s_n))}{n}\le\frac{c_2(T)}{n}\sum_{k=0}^{n-1}\rho_{\mu}^k(1+To_k(s_n)) \\
= \frac{c_2(T)(1-\rho_{\mu}^n)}{n(1-\rho_{\mu})}(1+To_n(s_n)) \underset {n \to \infty}{\longrightarrow} 0.
\end{gather*}
Combining the last relation, \eqref{eqfim3}, \eqref{eq68im} and Lemma \ref{l2} $(viii)$ yields the third case in \eqref{eq58i}.

Consider the case $i=2, \rho \searrow 1$. Relation \eqref{eq63im} and the fourth equation in \eqref{eqimm1} show that conditions of Lemma \ref{l15} are satisfied. We can express $v'(s_n-t_k(s_n))$ as $v'(s_n-\textbf{1}+\textbf{1}-\kappa +\kappa-t_k(s_n))$. Then, from \eqref{eq63im}, the fourth equation of \eqref{eqimm1}, \eqref{eqv3a}, we have
\begin{equation}\label{eqfim2a}
\sum_{k=0}^{n-1}v'(\textbf{1}-s_n+\kappa-\textbf{1})=\frac{Tn(\rho-1)}{\rho^{n-1}}(1+o_n)\underset {n \to \infty}{\longrightarrow} 0.
\end{equation}
Using \eqref{eqv}, the fourth equation in \eqref{eqimm1}, and Lemma \ref{l4}, we obtain
\begin{gather}
\nonumber \sum_{k=0}^{n-1}v'(\kappa-t_k(s_n))=\sum_{k=0}^{n-1}\frac{T(\rho-1)\rho^{k-n}}{1+QT\rho^{k-n}}(1+o_k(s_n))=\sum_{k=0}^{n-1}\frac{T(\rho-1)\rho^{k-n}}{1+QT\rho^{k-n}}(1+o_n) \\
\underset {n \to \infty}{\longrightarrow} \int_0^1\frac{Tdx}{1+QTx}=\frac{1}{Q}\ln{(1+QT)}. \label{eqfim2b}
\end{gather}
Combining \eqref{eqfim2a}, \eqref{eqfim2b}, \eqref{eq68im} and Lemma \ref{l2} $(viii)$ yields the fourth case in \eqref{eq58i}.
\end{proof}
\begin{remark}
Let $G^3_{\pm}(x,r)=\left(r^{\frac{\mp1}{2}}\frac{r^{1/2z_{\pm}(r,x)}+w_{\pm}(r,x)r^{-1/2z_{\pm}(r,x)}}{1+w_{\pm}(r,x)}\right)^{-\frac{\lambda'u}{Q}}$. Then, it is easy to check that $\lim_{r \to 1}G^3_{\pm}(x,r)=\left(\text{ch}\sqrt{2Qx}\right)^{-\frac{\lambda'u}{Q}}$, $\lim_{r \to 0}G^3_{+}(x,r)=\frac{1}{(1+{Qx})^{\frac{\lambda'u}{Q}}}$, $\lim_{r \to 0}G^3_{-}(x,r)=\frac{1}{e^{\lambda'ux}}$.
\end{remark}

Define
\[
m_n^4=
\begin{cases}
      1)\frac{\rho_{\mu}^{ n}-1 + n(1-\rho_{\mu})}{(1-\rho_{\mu})^2}v,i=1, \rho \gtrless 1, r < 1;\\
      2)\frac{n^2}{2}v, i=1, \rho  \gtrless 1, r=1;\\
      3)\frac{n}{1-\rho}v, i=2, \rho \nearrow 1,
    \end{cases}
\]
\[\Psi_2(r,T)=\begin{cases}
      1)\left(\text{ch}\sqrt{2QT}\right)^{-2}, i=1, r =1;\\
      2)\left(r^{\frac{1}{2}}\frac{r^{1/2z_{-}(r,T)}+w(r,T)r^{-1/2z_{-}(x,y)}}{1+w(r,T)}\right)^{-2},i=1, r < 1;\\
      3)e^{-2QT}, i=2.
         \end{cases}\]
Also, define
\[
P_{n,j}(x,s)=E\left(x^{Z_n}s^{Y_n}|Z_0=e_j\right), P_n(s,x)=\left(P_{n,1}(x,s),\ldots,P_{n,d}(x,s)\right).
\]
Similarly to the one-dimensional case (Pakes \cite[p.~189]{k8}), one can show that
\begin{gather}\label{eqas1o}
L_{n,j}(s,m)=E\left(s^{Y_n}|Z_{n+m}>0,N<\infty,Z_0=e_j\right)=\frac{P_{n,j}(\mu,s)-P_{n,j}(f_m(\textbf{0}),s)}{\mu_j-f_{m+n,j}(\textbf{0})},\\
\nonumber L_{n}(s,m)=\left(L_{n,1}(s,m),\ldots,L_{n,d}(s,m)\right)',m \in \mathbb{N}
\end{gather}
and 
\begin{equation}\label{eqas1}
P_{n,j}(x,s)=s_jf_j(P_{n-1}(x,s)).
\end{equation}
\begin{theorem}\label{t6}
Let $s_n=\left(e^\frac{-t_1}{m^4_{n,1}},\ldots,e^\frac{-t_d}{m^4_{n,d}}\right), t_k \ge0, k \in \{1,...,d\}$. Then for any $j \in \{1,...,d\}$
\begin{equation}\label{eq58iac}
\lim_{\sigma_1\to \infty}\lvert \lim_{m\to \infty}L_{n,j}(s_n,m)-\Psi_2(r,T)\rvert=0.
\end{equation}
\end{theorem}
\begin{proof}[Proof]
First, consider the case $\rho<1$. In this case $N<\infty$ is an almost sure event, $\mu=\textbf{1}$, and formula \eqref{eqas1o} can be rewritten as 
\[L_{n,j}(s,m)=\frac{P_{n,j}(\textbf{1},s)-P_{n,j}(f_m(\textbf{0}),s)}{{1}-f_{m+n,j}(\textbf{0})}=\frac{t_{n,j}(s)-P_{n,j}(f_m(\textbf{0}),s)}{{1}-f_{m+n,j}(\textbf{0})}.\]
Using representation \eqref{eq15}, \eqref{eqas1}, and the previous relation, we can write
\begin{gather*}
t_n(s)-P_n(f_m(\textbf{0}),s)=P_n(\textbf{1},s)-P_n(f_m(\textbf{0}),s)=s(f(P_{n-1}(\textbf{1},s))-f(P_{n-1}(f_m(\textbf{0}),s)))\\
=s\left(M(P_{n-1}(\textbf{1},s))-E({f_m(\textbf{0})})\right)\left(P_{n-1}(\textbf{1},s)-P_{n-1}(f_m(\textbf{0}),s)\right)\\
=\ldots= \prod_{k=0}^{n-1}s\left(M(P_{k}(\textbf{1},s))-E({P_{k}\left(f_m(\textbf{0}),s\right)})\right)\left(\textbf{1}-f_m(\textbf{0})\right).
\end{gather*}
Note that $v'\left(\textbf{1}-f_m(\textbf{0})\right)/v'\left(\textbf{1}-f_{m+n}(\textbf{0})\right)=\rho^{-n}\pi_m(\textbf{1}+o_m)/\pi_{m+n} \underset {m \to \infty}{\longrightarrow} \rho^{-n}$ by \eqref{eq14}. Also, because of the part $(2)$ of the Remark \ref{r1}, we can ignore $s$ in the last part of the above chain of equations by the same arguments as in Lemma \ref{l10}. Taking the limit of $L_{n}(m)$ as $m \to \infty$ and using \eqref{eq13a} yields
\begin{equation}\label{eqap}
L_n(s) = \lim_{m \to \infty}{L_n(s,m)}=\frac{\prod_{k=0}^{n-1}M(t_{k}(s))u}{\rho^nu}(1+o_n(s)).
\end{equation}
 Then,
\begin{equation}\label{eq63imm}
{1}-s_{n,j}=\frac{t_j}{m^4_n}(1+t_jo_n).
\end{equation}
Representation \eqref{eq18} for matrices $M(t_{k}(s_n))$ give
\begin{equation}\label{eq63imm2}
M(t_{k}(s_n)) = M+q\left(t_{k}(s_n), \textbf{1}\right)(1+o(s_n)).
\end{equation}
Since $m_n^4$ is a subset of $m_n^3$, from the proof of Theorem \ref{t5}, we deduce that $v'(1-s_n)/v'(s_n-t_k(s_n))=o(s_n), k \in \mathbb{N}$. Hence $\textbf{1}-t_{k}(s_n)=\textbf{1}-s_n+s_n-t_{k}(s_n)=v'(s_n-t_k(s_n))u(\textbf{1}+o_n(s_n))$ by \eqref{eq66im}.
Multiplying equation \eqref{eq63imm2} by $v'$ on the left-hand side and by $u$ on the right-hand side, and using Lemma \ref{l5}, we obtain
\begin{equation}\label{eqap1}
\rho_{t_{k}(s_n)}=\rho-2Qv'(s_n-t_{k}(s_n))(1+o_k(s_n)).
\end{equation}
According to Lemma \ref{l3}, $M^n(t_{k}(s_n))/\rho^n_{t_{k}(s_n)}= u_{t_{k}(s_n)}v'_{t_{k}(s_n)}(1+o_n)$ as $\sigma \to 0$. Since $t_{k}(s_n)=1+o(s_n), k \in \mathbb{N}$ by \eqref{eqin000}, \eqref{eq19}, and  \eqref{eqin1252}, $u_{t_{k}(s_n)}v'_{t_{k}(s_n)}=uv'(1+o(s_n))$. Then, it must be the case, that
\begin{equation}\label{eqap2}
\prod_{k=0}^{n-1}\frac{M(t_{k}(s_n))}{\rho_{t_{k}(s_n)}}=uv'(1+o_n(s_n)).
\end{equation}
Combining \eqref{eqap}, \eqref{eqap1}, \eqref{eqap2}, and considering that $(uv')u=u$, we derive
\begin{equation}\label{eqap3}
L_n(s_n) = \prod_{k=0}^{n-1}(1-2Qv'(s_n-t_{k}(s_n))(1+o_k(s_n)))(\textbf{1}+o_n(s_n)).
\end{equation}
Applying the Taylor expansion for the logarithm gives
\begin{gather}
\nonumber \ln{\prod_{k=0}^{n-1}(1-2Qv'(s_n-t_{k}(s_n))(1+o_k(s_n)))}=\sum_{k=0}^{n-1}\ln\left(1-2Qv'(s_n-t_{k}(s_n))(1+o_k(s_n))\right)\\
 =-2Q\sum_{k=0}^{n-1}v'(s_n-t_{k}(s_n))(1+o_k(s_n))+\sum_{k=0}^{n-1}R_{n,k},\label{eqap4}
\end{gather}
where the second sum tends to zero if the first is bounded.

The first sum in \eqref{eqap4} is analogous to the first sum in \eqref{eq68im} but with a different multiplier ($2Q$ instead of $\lambda'u$, which is also bounded by Lemma \ref{l2} $(iv)$ and $(v)$). The result follows from \eqref{eqap3}, \eqref{eqap4}, and the respective results in Theorem \ref{t5}.

In the case where $\rho>1$, consider the probability generating function $f^*(s)=f(\mu s)/\mu$ (see \cite{k9}), which by definition is the offspring probability generating function of a subcritical branching process with the Perron root $\rho_{\mu}$. Let $P^*(x,s)$ be the joint probability generating function of the population size and the total progeny of the $n$-th generation of this subcritical branching process. After dividing the numerator and the denominator of \eqref{eqas1o} by $\mu_j$, we obtain the formula
\begin{equation}\label{eqap5}
L^*_{n,j}(s,m)=\frac{P^*_{n,j}(\textbf{1},s)-P^*_{n,j}(f^*_m(\textbf{0}),s)}{{1}-f^*_{m+n,j}(\textbf{0})}=\frac{t^*_{n,j}(s)-P^*_{n,j}(f^*_m(\textbf{0}),s)}{{1}-f^*_{m+n,j}(\textbf{0})}.
\end{equation}
In view of \eqref{eq20}, we have  $\lim_{\sigma \to 0}\lvert n(1-\rho) \rvert = \lim_{\sigma \to 0}n(1-\rho_{\mu})=-\ln{r}$. Therefore, from \eqref{eqap5}, it is evident that the processes $L^*_{n,j}(m)$ will have the same asymptotics as the processes $L_{n,j}(m)$ in the subcritical case, with the same limits $\lim_{\sigma \to 0} n(1-\rho) = -\ln{r}$.

The theorem is proved.
\end{proof}

{\footnotesize

\end{document}